\documentclass[article,10pt,letterpaper]{article}

\usepackage[letterpaper,left=1in,right=1in,top=1.5in,bottom=1.5in]{geometry}

\title{Explicit descent on elliptic curves and splitting Brauer classes}
\author{Benjamin Antieau and Asher Auel}
\date{\today}

\usepackage[pdfstartview=FitH,
            pdfauthor={Benjamin Antieau and Asher Auel},
            pdftitle={Explicit descent on elliptic curves and splitting Brauer classes},
            colorlinks,
            linkcolor=reference,
            citecolor=citation,
            urlcolor=e-mail,
            backref]{hyperref}

\usepackage{amsmath}
\usepackage{amscd}
\usepackage{amsbsy}
\usepackage{amssymb}
\usepackage{verbatim}
\usepackage{eufrak}
\usepackage{eucal}
\usepackage{microtype}
\usepackage{mathrsfs}
\usepackage{amsthm}
\usepackage{stmaryrd}
\usepackage{xy}
\input xy
\xyoption{all}
\usepackage{tikz}
\usetikzlibrary{matrix,arrows,cd}
\usepackage{dsfont}
\usepackage[T1]{fontenc}

\usepackage{fancyhdr}
\usepackage{makeidx}
\makeindex

\usepackage[bbgreekl]{mathbbol}
\usepackage{amsfonts}
\DeclareSymbolFontAlphabet{\mathbb}{AMSb} 
\DeclareSymbolFontAlphabet{\mathbbl}{bbold}

\usepackage{color}
\definecolor{todo}{rgb}{1,0,0}
\definecolor{conditional}{rgb}{0,1,0}
\definecolor{e-mail}{rgb}{0,.40,.80}
\definecolor{reference}{rgb}{.20,.60,.22}
\definecolor{mrnumber}{rgb}{.80,.40,0}
\definecolor{citation}{rgb}{0,.40,.80}

\let\oldmarginpar\marginpar
\renewcommand\marginpar[1]{\-\oldmarginpar[\raggedleft\footnotesize #1]%
{\raggedright\footnotesize #1}}

\newcommand{\Ascr}{\mathcal{A}}

\newcommand{\Gscr}{\mathcal{G}}
\newcommand{\Hscr}{\mathcal{H}}

\newcommand{\Lscr}{\mathcal{L}}
\newcommand{\Mscr}{\mathcal{M}}

\newcommand{\Oscr}{\mathcal{O}}

\newcommand{\B}{\mathrm{B}}

\renewcommand{\d}{\mathrm{d}}
\newcommand{\E}{\mathrm{E}}

\renewcommand{\H}{\mathrm{H}}

\newcommand{\R}{\mathrm{R}}

\newcommand{\bC}{\mathbf{C}}

\newcommand{\bF}{\mathbf{F}}
\newcommand{\bG}{\mathbf{G}}

\newcommand{\bP}{\mathbf{P}}
\newcommand{\bQ}{\mathbf{Q}}
\newcommand{\bR}{\mathbf{R}}

\newcommand{\bZ}{\mathbf{Z}}

\renewcommand{\mathds}{\mathbbl}

\newcommand{\ZZ}{\mathds{Z}}

\newcommand{\pfrak}{\mathfrak{p}}

\newcommand{\xto}{\xrightarrow}

\DeclareMathOperator{\id}{id}

\newcommand{\im}{\mathrm{im}}

\renewcommand{\geq}{\geqslant}
\renewcommand{\leq}{\leqslant}

\DeclareMathOperator{\Pic}{Pic}
\DeclareMathOperator{\Br}{Br}

\newcommand{\Extscr}{\mathscr{E}\!\mathit{xt}}
\newcommand{\Homscr}{\mathscr{H}\!\mathit{om}}

\DeclareMathOperator{\Hom}{Hom}

\newcommand{\SL}{\mathrm{SL}}
\newcommand{\GL}{\mathrm{GL}}
\newcommand{\PGL}{\mathrm{PGL}}

\newcommand{\Gm}{\bG_{m}}

\newcommand{\per}{\mathrm{per}}

\newcommand{\ind}{\mathrm{ind}}

\DeclareMathOperator{\Spec}{Spec}

\newcommand{\we}{\simeq}
\newcommand{\iso}{\cong}

\theoremstyle{plain}
\newtheorem{theorem}{Theorem}[section]
\newtheorem*{theorem*}{Theorem}
\newtheorem{lemma}[theorem]{Lemma}

\newtheorem{proposition}[theorem]{Proposition}

\newtheorem{corollary}[theorem]{Corollary}
\newtheorem*{corollary*}{Corollary}

\newtheorem{namedtheorem}{Theorem}

\newtheorem{namedcorollary}[namedtheorem]{Corollary}

\theoremstyle{definition}

\newtheorem{namedexample}{Example}

\theoremstyle{definition}
\newtheorem{definition}[theorem]{Definition}
\newtheorem{warning}[theorem]{Warning}

\newtheorem{example}[theorem]{Example}
\newtheorem*{example*}{Example}

\newtheorem{question}[theorem]{Question}
\newtheorem*{question*}{Question}

\newtheorem{remark}[theorem]{Remark}
\newtheorem*{remark*}{Remark}

\newcommand{\linedef}{\textbf}

\begin{document}

\maketitle

\begin{abstract}
    \noindent
    We prove new results on splitting Brauer classes by genus $1$
    curves, settling in particular the case of degree $7$ classes over
    global fields.  Though our method is cohomological in nature, and
    proceeds by considering the more difficult problem of splitting
    $\mu_N$-gerbes, we use crucial input from the arithmetic of
    modular curves and explicit $N$-descent on elliptic curves.
%

\end{abstract}

\section{Introduction}

Inspired by work of Artin~\cite{artin:Brauer-Severi} on Severi--Brauer
varieties, Pete Clark~\cite{clark:open} and David
Saltman~\cite{rage:problems} have asked whether, given a Brauer class
$\alpha\in\Br(k)$ over a field $k$, there exists a (smooth proper
geometrically irreducible) genus $1$ curve $C/k$ such that $\alpha$ is
split by $C$.  By `split' we mean that $\alpha$ pulls back to zero in
$\Br(C)$, or equivalently, in $\Br(k(C))$.  We remark that if a curve
of genus $g$ splits $\alpha$, then the index of $\alpha$ must divide
$2g-2$, showing the relevance of the genus $1$ hypothesis. 

Work of Swets~\cite{swets} handles the case when $\alpha$ has index
$\leq 3$, de Jong and Ho~\cite{dejong-ho} settle the case when
$\alpha$ has index $\leq 5$, and the case of index $6$ is indicated in
\cite{auel:banff_talk}.  In a slightly different direction,
Roquette~\cite{roquette}, Lichtenbaum~\cite{lichtenbaum}, Ciperiani
and Krashen~\cite{ciperiani-krashen}, and Krashen and
Lieblich~\cite{krashen-lieblich} give results and algorithms to
compute the kernel of $\Br(k)\rightarrow\Br(C)$ when $C$ is a fixed
genus $1$ curve.  Using these results, one can establish an
affirmative answer to the question when $k$ is a local field: see
Example~\ref{ex:local}.  However, the question is still wide open,
notably over global fields.  One of our main contributions is the
following.

\begin{theorem*}
  Every Brauer class $\alpha$ of index 7 over a global field $k$ is
  split by a genus 1 curve.
\end{theorem*}

In fact, we prove much more by considering a strengthening of the
question: given a class $\beta\in\H^2(\Spec k,\mu_N)$, is there a
genus $1$ curve $C/k$ such that $\beta$ is split by $C$?  Here, by
`split' we mean that $\beta$ pulls back to zero in $\H^2(C,\mu_N)$,
and to emphasize this, we say that $C$ splits the $\mu_N$-gerbe $\beta$.  Note that this
is not equivalent to $\beta$ pulling back to zero in
$\H^2(\Spec k(C),\mu_N)$. Indeed, if $\beta$ is split by $C$, then the image
$\alpha \in \Br(k)$ under the isomorphism
$\H^2(\Spec k,\mu_N) \iso \Br(k)[N]$, is split by $C$, but the
converse is not generally true. {Unless otherwise specified,
cohomology groups are taken with respect to the fppf topology.}

For example, using the Tate pairing, one can show that if $k$ is a
non-archimedean local field and $\beta\in\H^2(\Spec k,\mu_N)\iso\bZ/N$,
then $\beta$ is split by a genus $1$ curve; see Proposition~\ref{prop:local}.

In our main theorem, we mention cyclic classes, i.e., those in the
image of the cup product map $(\chi,u)\mapsto\chi\cup u$ in cohomology
$\H^1(\Spec k, \bZ/N) \times \H^1(\Spec k, \mu_N) \to \H^2(\Spec k,
\mu_N)$. The theorem of Albert, Brauer, Hasse, and Noether says that if $k$ is
a global field, then every class of $\H^2(\Spec k,\mu_N)$ is cyclic. This fact is
central to the proof of the following.

\begin{namedtheorem}\label{thm:a}
  Let $k$ be a field and $\beta\in\H^2(\Spec k,\mu_N)$. Then the
  $\mu_N$-gerbe $\beta$ is split by a genus $1$ curve in the following
  cases:
    \begin{itemize}
        \item $N=2,3,4,5$ and $\beta$ is cyclic,
        \item $N=6,7,10$ and $k$ is a global field,
        \item $N=8$, $k$ is a global field, and $k$ contains a
            primitive $8$th root of unity,
        \item $N=9$, $k$ is a global field, and $k$ contains
          $\zeta_9+\zeta_9^{-1}$, where $\zeta_9$ is a primitive $9$th
          root of unity, or
        \item $N=12$, $k$ is a global field, and $k$ contains a primitive
            $4$th root of unity.
    \end{itemize}
\end{namedtheorem}

The reader will observe that the $N$ appearing in Theorem~\ref{thm:a}
are precisely those for which the modular curve $X_1(N)$ has genus $0$
(see for example~\cite{ogg}). This is not a coincidence; the proof
uses explicit $1$-parameter families of elliptic curves with exact
order $N$ points constructed using Tate normal form; see the work of
Kubert~\cite{kubert} and the reference by Knapp~\cite[V.5]{knapp} as
well as the tables by Sutherland~\cite{sutherland,sutherland:table}.
{\ttfamily MAGMA}~\cite{magma} code provided by Tom Fisher allows us
to do explicit $\mu_N$-descent on these curves, which we do to
generate lots of classes $\beta$ split by genus $1$ curves. Saltman
led us to the cyclicity theorem of Albert~\cite{albert}, generalized
to non-prime-degree algebras by Vishne~\cite{vishne} and
Min\'{a}\v{c}--Wadsworth~\cite{minac-wadsworth}, which allows us to
prove that in the global cases these classes span the entire group
$\H^2(\Spec k,\mu_N)$ under the assumptions in the theorem.

While our results for splitting $\mu_N$-gerbes, covering small $N$ and
particular fields $k$, might seem limited, there is little chance that
they can be improved in general.
Indeed, over $\bQ$, no $\mu_p$-gerbe 
is split by a genus $1$ curve if $p\geq 11$ is a prime (see
Example~\ref{ex:bigprimes}), although the question of whether classes
in $\Br(\bQ)[p]$ are split by genus 1 curves is still
open. Additionally, we remark that the splitting problem for
$\mu_N$-gerbes is sensitive to $N$: in Example~\ref{ex:mnotn}, we give
an example of an index $2$ Brauer class $\alpha$ split by a genus $1$
curve $X$, where $X$ splits the unique lift of $\alpha$ to
$\H^2(\Spec k,\mu_4)$ but does not split the unique lift to
$\H^2(\Spec k,\mu_2)$.

As an additional illustration of our methods, we also prove that every
class $\beta\in\H^2(\Spec k,\mu_N)$ is split by a torsor for an
abelian variety; see Section~\ref{sec:abelian}. This gives a new proof
of the fact that every Brauer class is split by an abelian variety
torsor, which was first proved by Ho and Lieblich
in~\cite{ho-lieblich}. They prove even more, namely that every Brauer
class is split by a torsor for the Jacobian of a curve, typically of
very high genus.  The curves they employ have simple Jacobians and
they wonder whether one can split Brauer classes by products of genus
one curves. We use our methods to address this question in many cases
in Theorem~\ref{thm:b} and~\ref{thm:c}.

\begin{namedtheorem}\label{thm:b}
  Let $k$ be a field and let $N\geq 2$ be an integer invertible in
  $k$. If $\beta\in\H^2(\Spec k,\mu_N)$, then the $\mu_N$-gerbe
  $\beta$ is split by a product of genus $1$ curves in the following
  cases:
    \begin{itemize}
        \item   $N=2,3,5,6,10,15,30$ or
        \item   $N=4,12,20,60$ if $k$ contains a primitive $4$th root of unity.
    \end{itemize}
\end{namedtheorem}

Clark first suggested in unpublished work~\cite{clark:open} a close
connection between the splitting problem and the period-index
obstruction map for genus $1$ curves studied by O'Neil~\cite{oneil}
and Clark~\cite{clark-wc1}. We use those ideas to prove the following
theorem, clarifying along the way the relationship between the
period-index map and the cup product. The theorem applies to cyclic
Brauer classes, i.e., the Brauer classes associated to cyclic
$\mu_N$-cohomology classes.

\begin{namedtheorem}\label{thm:c}
  Let $k$ be a field and fix an elliptic curve $E$ over $k$ with a
  full level $N$ structure $E[N]\iso\bZ/N\times\mu_N$. If
  $\alpha\in\Br(k)[N]$ is cyclic, then $\alpha$ is split by an
  $E$-torsor.
\end{namedtheorem}

Saltman~\cite[Corollary~2]{saltman:cubic} has given another proof of
the $N=3$ case of Theorem~\ref{thm:c}, and also gives examples to show
that there are obstructions to splitting cubic classes by genus
$1$ curves with Jacobians of given $j$-invariant.

It is always possible to arrange for the existence of an elliptic
curve with a full level $N$ structure as above after a small Galois
extension of $k$. So, we obtain our main corollary, which says that,
after some small extensions, we can split arbitrary Brauer classes by
products of genus $1$ curves while controlling the $j$-invariants of
their Jacobians. Note that in the contexts of Theorems~\ref{thm:a}
and~\ref{thm:b}, there is no control of the $j$-invariants of the
Jacobians.

\begin{namedcorollary}\label{cor:d}
  Let $k$ be a field and let $N>1$ be an integer. Fix an elliptic
  curve $E$ over $k$, which is non-supersingular if the characteristic
  of $k$ divides $N$. There is a finite extension $K$ over $k$ of
  degree dividing the order of $\GL_2(\bZ/N)$ such that every
  $\alpha\in\Br(K)[N]$ is split by a product of $E$-torsors.
\end{namedcorollary}

The extension $K$ is obtained by adjoining the coordinates of the
$N$-torsion points of $E$, with suitable modifications if the
characteristic divides $N$.  The number of torsors needed to split
$\beta$ or $\alpha$ in Theorems~\ref{thm:b} and Corollary~\ref{cor:d}
is dictated by the symbol lengths of $\beta$ or $\alpha$ when
expressed as sums of cyclic classes using the theorems of
Teichm\"uller (when $k$ has characteristic $p$ and $N$ is a $p$-power,
see~\cite[Thm.~9.1.4]{gille-szamuely}),
Merkurjev~\cite{merkurjev-norm,merkurjev-fields} ($N=2,3$),
Matzri~\cite{matzri} ($N=5$), and
Merkurjev--Suslin~\cite{merkurjev-suslin} (when $N$ is invertible in
$k$, noting that the extension $K$ in Corollary~\ref{cor:d} will
contain a primitive $N$th root of unity).

We mention a series of examples where Theorem~\ref{thm:c} and
Corollary~\ref{cor:d} can be applied.

\begin{namedexample}
  In the setting of Hilbert's 12th problem (Kronecker's Jugendtraum),
  suppose that $k$ is an imaginary quadratic number field $k$ with
  class number $1$ and $E$ is an elliptic curve over $k$ with complex
  multiplication in $k$. If
  $K=k^\mathrm{ab}=k(E_{\mathrm{tors}}(\overline{k}))$ is the maximal
  abelian extension (see~\cite[Chap.~II]{silverman-advanced}) obtained
  by adjoining the coordinates of the torsion points of $E$, then
  every cyclic algebra over any field $L$ containing $K$ is split by
  an $E$-torsor and every Brauer class over any such field $L$ is
  split by a product of $E$-torsors.
\end{namedexample}

\begin{namedexample}
  If $E$ is any elliptic curve defined over $\overline{\bQ}$, then
  every cyclic algebra over any field $L$ containing $\overline{\bQ}$
  is split by an $E$-torsor and every Brauer class over any such field
  $L$ is split by a product of $E$-torsors. This gives examples of the
  phenomenon with arbitrary algebraic $j$-invariant.
\end{namedexample}

\begin{namedexample}
  In characteristic $p$, no extension is necessary when $N=p$.  By a
  theorem of Deuring~\cite{deuring}, also called the
  Hasse--Deuring--Waterhouse theorem (see~\cite{waterhouse}), there is
  an elliptic curve $E$ over $\bF_p$ with exactly $p$ rational points.
  In particular, it is ordinary and admits a full level $p$ structure.
  This fact can also be seen by the more general Honda--Tate
  theory~\cite{tate-injective,honda,tate-surjective}; the elliptic
  curve corresponds to the Weil polynomial $T^2-T+p$ and is unique up
  to isogeny. Therefore, for any field $k$ of characteristic $p$,
  every cyclic class of $\Br(k)[p]$ is split by an $E$-torsor, and by
  the above mentioned theorem of Teichm\"uller
  (see~\cite[Thm.~9.1.4]{gille-szamuely}), every class of $\Br(k)[p]$
  is a sum of cyclic algebras and hence is split by a
  product of $E$-torsors.
\end{namedexample}

\begin{namedexample}
  If $E$ is any ordinary elliptic curve defined over
  $\overline{\bF}_p$, then $E$ admits a full level $N$ structure for
  any $N$.  Therefore, for any field $L$ containing
  $\overline{\bF}_p$, every cyclic class in $\Br(L)$ is split by an
  $E$-torsor.  However, by a theorem of Albert
  (see~\cite[Thm.~9.1.8]{gille-szamuely}), every class in
  $\Br(k)[p^\infty]$ is cyclic.  This proves that any $p$-primary
  Brauer class over a field containing $\overline{\bF}_p$ is split by
  a genus $1$ curve.
\end{namedexample}

Finally, we provide two additional applications of our methods to
splitting Galois cohomology classes of higher degree by products of
genus one curves.

\begin{namedcorollary}\label{cor:e}
  Let $k$ be a field, let $N \mid 60$ and assume that $k$ contains a
  primitive $N$th root of unity. For any degree $n \geq 2$, any
  $\beta\in\H^n(\Spec k,\mu_N)$ is split by a product of
  genus $1$ curves. The same result holds for $N$-torsion classes of
  $\H^n(\Spec k,\Gm)$.
\end{namedcorollary}

\begin{proof}
  By the Bloch--Kato conjecture, we can write
  $\beta=\beta_1+\cdots+\beta_d$ as a sum of $d$ symbols
  $\beta_i=u_{i1}\cup\cdots\cup u_{in}$, where each
  $u_{ij}\in k^\times/(k^\times)^N$, and where we use a primitive
  $N$th root of unity to give an isomorphism $\bZ/N\iso\mu_N$. Each
  $u_{i1}\cup u_{i2}\in\H^2(\Spec k,\mu_N)$ can be split by a product
  of genus $1$ curves by Theorem~\ref{thm:b}. But, this implies that
  each $\beta_i$ is split by a product of genus $1$ curves and hence
  so is $\beta$.
\end{proof}

Splitting higher degree Galois cohomology classes, combined with the
Milnor conjecture for the Witt group as proved by
Voevodsky~\cite{voevodsky:Milnor_conjecture_I},
\cite{voevodsky:Milnor_conjecture_II} and
Orlov--Vishik--Voevodsky~\cite{orlov_vishik_voevodsky}, we arrive
at the following application to the study of Witt kernels for torsors
under abelian varieties.

\begin{namedcorollary}\label{cor:f}
  Let $\sigma$ be a quadratic form of even dimension $d \geq 4$ and trivial
  discriminant over a field $k$ of characteristic $\neq 2$. Then,
  $\sigma$ becomes hyperbolic after extension to the function field of
  a product of genus $1$ curves.
\end{namedcorollary}

\begin{proof}
  A consequence of the Arason--Pfister Hauptsatz and the Milnor
  conjectures for the Witt group is the statement that a quadratic
  form $\sigma$ is hyperbolic if and only if $e_n(\sigma)=0$ for all
  $n \geq 0$, where $e_n : I^n(k) \to \H^n(k,\mu_2)$ are the higher
  cohomological invariants on the fundamental filtration of the Witt
  group, see \cite[Sections~16,~23.A]{elman_karpenko_merkurjev}.  By
  successively splitting $e_n(\sigma)$ and then subtracting off a sum
  of Pfister forms representing $e_n(\sigma)$, it follows that
  $\sigma$ will become hyperbolic over any field extension that splits
  a certain finite collection of mod 2 Galois cohomology classes.  The
  hypotheses on $\sigma$ imply that $e_0(\sigma)=e_1(\sigma)=0$, and
  then Corollary~\ref{cor:e} applies to show that we can split any
  finite collection of Galois cohomology classes of degree~$\geq 2$ by
  a product of genus $1$ curves.
\end{proof}

We remark that the conditions on the dimension and discriminant are
necessary.  Whether a quadratic form in $I^2(k)$ becomes hyperbolic
over the function field of a genus $1$ curve is an open question.

\paragraph{Outline.} Section~\ref{sec:cohomology} mostly contains
background on the cohomological pairings that will be important for
our work, emphasizing their stack-theoretic development and
interpretation.  Readers very well-acquainted with the Weil and Tate
pairings could reasonably skip Sections~\ref{subsec:Tate}
and~\ref{subsec:Weil}, but should look at
Section~\ref{sec:clark-oneil} on the period-index obstruction theory
for genus $1$ curves, which contains the main tools going into proving
Theorem~\ref{thm:c}.  Section~\ref{sec:proofs} sets up the general
problem of splitting $\mu_N$-gerbes and contains proofs of the main
theorems.  In particular, Section~\ref{subsec:splitting-mu_N} includes
many examples showing the difference between the problem of splitting
Brauer classes and splitting $\mu_N$-gerbes, as well as complete
solutions over nonarchimedean local fields and the real numbers, while
Section~\ref{subsec:mu_n-isogeny} explains the relationship between
$N$-isogenies and splitting genus $\mu_N$-gerbes.  Finally,
Appendix~\ref{app} contains the explicit descent calculations we use
in the proof of Theorem~\ref{thm:a}.

\paragraph{Acknowledgments.} We benefited from conversations with Pete
Clark, Skip Garibaldi, David Gepner, Danny Krashen, Max Lieblich,
Lennart Meier, Anthony V\'arilly-Alvarado, Bianca Viray, and John
Voight and we thank them for their insights. Special thanks go to Tom
Fisher who provided \texttt{MAGMA} code and David Saltman for a
pivotal observation. Finally, Pete Clark, Cathy O'Neil, David Saltman,
and John Voight gave us invaluable feedback and pointers on a
preliminary version of this paper.

BA was supported by the NSF under grants DMS-2102010 and DMS-2120005
and by a Simons Fellowship. He would like to thank the departments at
Dartmouth and Yale for their hospitality during visits where this work
was carried out.  AA was supported by a Simons Foundation
Collaboration Grant and a Walter and Constance Burke Research Award.

Work on this project began at the workshop ``The use of linear
algebraic groups in geometry and number theory'' at Banff in September
2015. We thank the organizers, Garibaldi, Lemire, Parimala, and
Zainoulline, for creating the occasion and BIRS for providing a
wonderful setting for research.

\section{Cohomology and pairings}\label{sec:cohomology}

We give some background on the Tate and Weil pairings, cyclic algebras, and the
O'Neil period-index obstruction map~\cite{oneil}. This work culminates in
Proposition~\ref{prop:oneil}, which gives the connection between the period-index
obstruction map and cyclic algebras, generalizing prior work of
O'Neil~\cite{oneil} and Clark~\cite{clark-wc1}.

\subsection{The Tate pairing}
\label{subsec:Tate}

Mumford attributes the following statement to Rosenlicht and Serre;
see~\cite[p. 227]{mumford-abelian}.  Let $S$ be an arbitrary base scheme.

\begin{proposition}[Rosenlicht--Serre]
    Let $p\colon A\rightarrow S$ be an abelian scheme. The fppf Ext sheaf
    $\Extscr^1(A,\Gm)$ is naturally isomorphic to the dual abelian scheme
    $\widehat{A}=\Pic^0_{A/S}$, while $$\Extscr^0(A,\Gm)=0.$$
\end{proposition}

\begin{proof}
    We use a technique from crystalline cohomology, namely a resolution
    $$\cdots\rightarrow\bZ[A^3]\oplus\bZ[A^2]\rightarrow\bZ[A^2]\rightarrow\bZ[A]\rightarrow
    A$$ in fppf sheaves of abelian groups (see~\cite{berthelot-breen-messing} for details). We will only need that the
    homomorphism $\bZ[A^2]\rightarrow\bZ[A]$ is given by
    $m^*-p_1^*-p_2^*$, where $m : A^2\rightarrow A$ is the group law
    on $A$ and $p_i : A^2 \to A$ are the projections. Taking Ext
    sheaves out of this resolution into $\Gm$, we obtain a spectral
    sequence converging to $\Extscr^*(A,\Gm)$.  Because each
    term $A^n$ appearing is proper and geometrically connected, the
    $0$-line given by $\Extscr^0(\bZ[A^n],\Gm)\iso\R^0p_*\Gm$ is
    $\Gm\xrightarrow{\id}\Gm\rightarrow\Gm\times\Gm\rightarrow\cdots$.
    The $1$-line is given by $\R^1p_*\Gm$ from the various copies of
    $A$, which gives
    $\Pic_{A/S}\rightarrow\Pic_{A^2/S}\rightarrow\cdots$. Hence, we see that
    $\Extscr^0(A,\Gm)=0$ and that $\Extscr^1(A,\Gm)$ is
    given by the kernel of the
    map
    $$\Pic_{A/S}\xrightarrow{\Lscr\mapsto m^*\Lscr\otimes
      p_1^*\Lscr^{-1}\otimes p_2^*\Lscr^{-1}}\Pic_{A^2/S}.$$ It is
    part of the construction of the dual abelian variety, cf.\
    \cite[\textsection 13]{mumford-abelian}, that this
    kernel is the connected component $\widehat{A}=\Pic^0_{A/S}$ of the identity.
\end{proof}

For any $n$, the Yoneda product (composition) induces a bilinear pairing
$$\widehat{A}\times\B^n A\rightarrow\B^{n+1}\Gm$$ of commutative group
stacks. When $n=0$, we obtain $\widehat{A}\times A\rightarrow\B\Gm$,
which classifies the Poincar\'e line bundle (cf.\
\cite[Theorem~9.14]{huybrechts:Fourier-Mukai}), while for $n=1$ we
obtain $\widehat{A}\times \B A\rightarrow\B^2\Gm$, which defines the
bilinear Tate pairing.

\begin{definition}
  The \linedef{Tate pairing} will refer to the bilinear pairing on cohomology
    $$[-,-]_T\colon\widehat{A}(S)\times\H^1(S,A)\rightarrow\H^2(S,\Gm)$$
  induced from $\widehat{A}\times \B A\rightarrow\B^2\Gm$ as above.
\end{definition}

Here are two different perspectives on the Tate pairing,
cf.~\cite[Thm.~3.7]{ciperiani-krashen} and the citations there. The
following interpretations of the Tate pairing can also be easily
deduced from the functoriality of the Yoneda product and $\Extscr^*$.

First, there is the semiabelian point of view. Sections $\Lscr$ of
$\widehat{A}\iso\Extscr^1(A,\Gm)$ classify semiabelian $S$-scheme
extensions $A_\Lscr$ of the form
$$1\rightarrow\Gm\rightarrow A_\Lscr\rightarrow A\rightarrow 0.$$
Pairing with $\Lscr$ induces a homomorphism that on $S$-points
corresponds to the boundary maps
$$[\Lscr,-]_T\colon\H^n(S,A)\rightarrow\H^{n+1}(S,\Gm),$$ 
in the long exact sequence associated to the extension
$A_\Lscr$. Thus, if $X\in\H^1(S,A)$, then the class
$[\Lscr,X]_T\in\H^2(S,\Gm)$ is the obstruction to lifting the $A$-torsor
$X$ to an $A_\Lscr$-torsor.

Second, there is the Leray--Serre point of view. Fixing an $A$-torsor $X$, the
connected component $\Pic_{X/S}^0$ of the identity of $\Pic_{X/S}$ is
isomorphic to $\widehat{A}$. The Leray--Serre spectral sequence
$$\E_2^{s,t}=\H^s(S,\R^tp_*\Gm)\Longrightarrow\H^{s+t}(X,\Gm)$$
has a $d_2$-differential giving a homomorphism
$$d_2^X\colon\Pic_{X/S}(S)\iso\H^1(S,\R^1p_*\Gm)\rightarrow\H^2(S,\R^0p_*\Gm)\iso\H^2(S,\Gm).$$
When restricted to $\Lscr\in\Pic_{X/S}^0(S)$, we have $[\Lscr,X]_T=d_2^X(\Lscr).$
When $S=\Spec k$ is the spectrum of a field, this differential is the
obstruction to lifting the Galois-invariant divisor class
$\Lscr\in\Pic^0_{X/k}(S)$ to a line bundle on $X$.

The relevance of the Tate pairing to the problem of splitting Brauer
classes by torsors under abelian schemes stems from the fact that,
according to the Leray--Serre perspective, an $A$-torsor $X$ splits
any class of the form $[\Lscr,X]_T \in \H^2(S,\Gm)$, cf.\
\cite[Section~3.1]{ciperiani-krashen}.

We illustrate how to use the Tate pairing to show that every Brauer class on a
local field is split by a genus $1$ curve.

\begin{theorem}[Tate~\cite{tate-wc}, Shatz~\cite{shatz},
    Milne~\cite{milne-wc,milne-wc-addendum}]
    If $K$ is a non-archimedean local field and we equip
    $\widehat{A}(K)$ with the analytic topology and $\H^1(\Spec K,A)$
    with the discrete topology, then the Tate pairing $$[-,-]_T\colon
    \widehat{A}(K)\times\H^1(\Spec K,A)\rightarrow\Br(K)\iso\bQ/\bZ$$ is continuous
    and non-degenerate.
\end{theorem}

\begin{example}\label{ex:local}
  Let $K$ be a non-archimedean local field with uniformizer
  $\varpi$. The elliptic curve $E$ over $\Spec K$ given by
  $y^2+xy=x^3+\varpi^N$ has split multiplicative reduction and, by a
  result of Kodaira and N\'eron (see~\cite[Thm.~VII.6.1]{silverman}
  or~\cite[Sec.~1.5]{blr-neron}), $E(K)/E_0(K)\iso\bZ/N$, where
  $E_0(K)\subseteq E(K)$ is the group of points of smooth
  reduction. It follows from the non-degeneracy of the Tate pairing
  that, for any generator $P$ of $E(K)/E_0(K)$, there exists an
  $E$-torsor $X$ such that $[P,X]_T$ generates
  $\tfrac{1}{N}\bZ/\bZ\subseteq\bQ/\bZ$.\footnote{Here and elsewhere
    we will silently identify $E$ with $\widehat{E}$ via the
    polarization arising from the ample line bundle $\Oscr(0_E)$.} In
  particular, every class of $\Br(K)\iso\bQ/\bZ$ is split by a genus
  $1$ curve.
\end{example}

\subsection{The Weil pairing}
\label{subsec:Weil}

Let $\varphi\colon A\rightarrow A'$ be an isogeny of abelian
$S$-schemes with kernel $\ker(\varphi)$. Taking $\Extscr^*$ into
$\Gm$, we obtain an exact sequence
$$0\rightarrow\Extscr^0(\ker(\varphi),\Gm)\rightarrow\Extscr^1(A',\Gm)\rightarrow\Extscr^1(A,\Gm)\rightarrow\Extscr^1(\ker(\phi),\Gm).$$
The central two terms can be rewritten as
$\widehat{A'}\xrightarrow{\widehat{\varphi}}\widehat{A}$, which is the
dual isogeny.  Since isogenies are surjective as fppf sheaves, we
obtain a short exact sequence
$$0\rightarrow\widehat{\ker(\varphi)}\rightarrow\widehat{A'}\xrightarrow{\widehat{\varphi}}\widehat{A}\rightarrow
0,$$ where $\widehat{\ker(\varphi)}$ is the usual Cartier dual. In particular, we obtain a canonical isomorphism between the Cartier dual
of $\ker(\varphi)$ and of $\ker(\widehat{\varphi})$. In other words, there is a
canonical perfect pairing
$$\ker(\varphi)\times\ker(\widehat{\varphi})\rightarrow\Gm.$$ If
$\ker(\varphi)\subseteq A[N]$, e.g., if $\varphi$ is an $N$-isogeny,
then this pairing lands in $\mu_N$.

\begin{example}
    If $\varphi$ is multiplication by $N$, we obtain the Weil pairing
    $$A[N]\times\widehat{A}[N]\rightarrow\mu_N.$$ In the special case of an
    elliptic curve $E$ (or a principally polarized abelian variety), we can
    rewrite this as a perfect pairing $E[N]\times E[N]\rightarrow\mu_N$.
\end{example}

\begin{definition}\label{def:weil}
  While there are several possible versions of a
  \linedef{cohomological Weil pairing}, we will need the following two
  on an elliptic curve $E$ over $S$.  If $\varphi$ is an $N$-isogeny,
  then the Weil pairing associated to $\varphi$ is the bilinear pairing
    $$(-,-)_\varphi\colon\H^1(S,\ker(\varphi))\times\H^1(S,\ker(\widehat{\varphi}))\rightarrow\H^2(S,\mu_N)$$
    taking values in $\mu_N$-cohomology.  Taking multiplication by
    $N$, the Weil pairing is the bilinear pairing
    $$[-,-]_N\colon\H^1(S,E[N])\times\H^1(S,E[N])\rightarrow\H^2(S,\Gm)$$
taking values in the Brauer group.
\end{definition}

\begin{remark}
  The following compatibility between the Tate and Weil pairings will
  be useful to us.  Suppose that $X\in\H^1(S,E)[N]$, i.e, an
  $N$-torsion element of $\H^1(S,E)$. Pairing with $X$ defines a
  homomorphism $[-,X]_T\colon E(S)\rightarrow\H^2(S,\Gm)$ that factors
  through $E(S)/NE(S)\subseteq\H^1(S,E[N])$. If we lift $X$ to an
  $E[N]$-torsor $Y$, then we have that for $\Lscr\in E(S)$
    $$[\Lscr,X]_T=[\overline{\Lscr},Y]_N$$ in $\H^2(S,\Gm)$, where
    $\overline{\Lscr}$ is the image of $\Lscr$ in $E(S)/NE(S)$.
\end{remark}

\subsection{The cup product pairing and cyclic algebras}
\label{sec:cyclic}

The groups $\bZ/N$ and $\mu_N$ are Cartier dual.  Under the natural
identification $\mu_N = \Homscr(\bZ/N,\Gm)$, the bilinear evaluation
pairing $\bZ/N \times \mu_N \to \Gm$ takes the form
$(a,\zeta) \mapsto \zeta^a$, hence takes values in $\mu_N$.  It
induces a canonical isomorphism of sheaves $\bZ/N \otimes \mu_N \iso \mu_N$. Thus
we have a bilinear cup product pairing
$$
(-)\cup(-) \colon \H^1(S,\bZ/N) \times \H^1(S,\mu_N) \to \H^2(S,\bZ/N \otimes \mu_N) \to \H^2(S,\mu_N)
$$
and classes of the form $\chi\cup u$ in $\H^2(S,\mu_N)$ are called
cyclic.\footnote{Note that we write $u$ for a general element of
$\H^1(S,\mu_N)$, although for a general scheme $u$ is not necessarily a unit
modulo $N$th powers. Rather, there is a natural exact sequence
$$0\rightarrow\Oscr(S)^\times/(\Oscr(S)^\times)^N\rightarrow\H^1(S,\mu_N)\rightarrow\Pic(S)[N]\rightarrow
0.$$}

This pairing is related to the cyclic algebra construction as follows. Consider the nontoral embedding
$\bZ/N \times \mu_N \hookrightarrow \PGL_N$, which on points sends
$(1,1)$ to the $N \times N$-matrix
$$
\begin{pmatrix}
        0   &   0   &   \cdots  &   0   &   1\\
        1   &   0   &   \cdots  &   0   &   0\\
        0   &   1   &   \cdots  &   0   &   0\\
        \vdots  &   \vdots &   \ddots & \vdots    &   \vdots\\
        0   &   0   &   \cdots  &   1   &   0\\
    \end{pmatrix}
$$
and $(0,\zeta)$ to the diagonal matrix $(1, \zeta, \dotsc,
\zeta^{N-1})$.  This embedding induces a map on cohomology
$$
\H^1(S,\bZ/N \times \mu_N) = \H^1(S,\bZ/N) \times \H^1(S,\mu_N) \to \H^1(S,\PGL_N)
$$
which to $\chi \in \H^1(S,\bZ/N)$ and $u \in \H^1(S,\mu_N)$ associates
an Azumaya algebra $\Ascr(\chi,u)$ of degree $N$ on $S$, which is
called the cyclic algebra associated to the pair $(\chi,u)$.  When
$S = \Spec k$ is the spectrum of a field, this coincides with the
classical cyclic algebra construction by
\cite[Props.~2.5.2,~4.7.3]{gille-szamuely}.  Composing with the
natural coboundary Brauer class map $\H^1(S,\PGL_N) \to \H^2(S,\Gm)$,
we arrive at the well-known result that the Brauer class of the cyclic
algebra $\Ascr(\chi,u)$ in $\H^2(S,\Gm)$ coincides with the image,
which we denote by
$[\chi,u]$, of the cup product $\chi\cup u$ under the natural map
$\H^2(S,\mu_N) \to \H^2(S,\Gm)$.

We now give a stack-theoretic interpretation, building on
Lieblich~\cite[Sec.~4.3]{lieblich-surface}.  The cup product pairing
on $\H^1(S,\bZ/N \times \mu_N) = \H^1(S,\bZ/N) \times \H^1(S,\mu_N)$,
followed by the natural map $\H^2(S,\mu_N) \to \H^2(S,\Gm)$, gives
rise to a functorial assignment $(\chi,u)\mapsto[\chi,u]$ on cohomology
$\H^1(S,\bZ/N \times \mu_N) \to \H^2(S,\Gm)$, which in turn
corresponds to a Brauer class on the classifying stack
$\B (\bZ/N \times \mu_N)$.

Using the Leray spectral sequence in flat
cohomology, one can derive a decomposition
\begin{equation}
\begin{split}
\H^2(\B (\bZ/N \times \mu_N),\Gm) &{}\iso\bZ/N\oplus\H^1(S,\Pic_{\B(\bZ/N \times \mu_N)/S})\oplus\H^2(S,\Gm)\label{eq:brB}\\
&{}\iso\bZ/N\oplus\H^1(S,\bZ/N \times \mu_N)\oplus\H^2(S,\Gm).
\end{split}
\end{equation}
Here we have utilized isomorphisms
$\Pic_{\B(\bZ/N \times \mu_N)/S} \iso \Homscr(\bZ/N \times
\mu_N,\Gm) \iso \bZ/N \times \mu_N$, where the first is canonical
and the second follows from Cartier duality.

A computation similar to \cite[Sec.~4.3]{lieblich-surface} then shows that
the Brauer class $\kappa \in \B (\bZ/N \times \mu_N)$, arising from the cup
product pairing, is a generator of the $\bZ/N$ factor in the decomposition \eqref{eq:brB}.

\subsection{The obstruction theory of Clark and O'Neil}
\label{sec:clark-oneil}

Let $\Lscr$ be a relatively ample line bundle of degree $N$ on an
elliptic curve $p\colon E\to S$, so that $\Lscr$ defines an
$S$-embedding $E\hookrightarrow\bP(p_*\Lscr)$, where $p_*\Lscr$ is a
rank $N$ vector bundle on $S$ by Riemann--Roch. The $S$-scheme of
$S$-automorphisms $\PGL_N^\Lscr$ of $\bP(p_*\Lscr)$ is a typically
non-split form of $\PGL_N$ given specifically by $\GL_N^\Lscr/\Gm$,
where $\GL_N^\Lscr$ is the sheaf of automorphisms of the vector bundle
$p_*\Lscr$ over $S$. If $p_*\Lscr$ is trivial, then these forms will
be split. The subgroup of $\PGL_N^\Lscr$ coming from translations of
the elliptic curve agrees with $E[N]$. Indeed, if $x\in E(T)$ is a
section over some $S$-scheme $T$, then $x^*\Lscr\iso\Lscr$ if and only
$x\in E[N](T)$. It follows that sections of $E[N]$ extend to give
automorphisms of $\bP(p_*\Lscr)$. The induced map
$E[N]\hookrightarrow\PGL_N^\Lscr$ is a typically twisted form of the
standard non-toral cyclic subgroup $\bZ/N\times\mu_N\subseteq\PGL_N$
utilized in Section~\ref{sec:cyclic}. Note that it is twisted both in
the sense that $E[N]$ is only \'etale-locally isomorphic (in
the non-supersingular case) to $\bZ/N\times\mu_N$ and it is also
twisted in the sense that even given a full level structure and a
trivialization of $p_*\Lscr$, the associated embedding
$\bZ/N\times\mu_N\hookrightarrow\PGL_N$ is not {\em a priori} conjugate to the
standard non-toral cyclic subgroup over $S$.

In any case, by pulling back, we obtain a commutative diagram of
central extensions
\begin{equation}\label{eq:theta}
    \begin{gathered}
        \xymatrix{
            1\ar[r]&\Gm\ar@{=}[d]\ar[r]&\Gscr_\Lscr\ar[r]\ar[d]&E[N]\ar[d]\ar[r]&0\\
            1\ar[r]&\Gm\ar[r]&\GL_N^\Lscr\ar[r]&\PGL_N^\Lscr\ar[r]&0,
        }
    \end{gathered}
\end{equation}
where $\Gscr_\Lscr$ is a nonabelian affine algebraic group scheme, the
theta group of Mumford associated to the pair $(E,\Lscr)$;
see~\cite{mumford-abelian-1}.

\begin{definition}
    Associated to the theta group $\Gscr_\Lscr$ is the {\bf period-index
    obstruction map}
    $$Ob_\Lscr\colon\H^1(S,E[N])\rightarrow\H^2(S,\Gm),$$
    which is the boundary map in non-abelian cohomology associated to
    the short exact sequence in the top row of~\eqref{eq:theta}.  We
    will write $Ob$ for the special case when $\Lscr=\Oscr(N 0_E)$ is
    the line bundle associated to the origin of $E$ with multiplicity
    $N$.
\end{definition}

The period-index obstruction map gets its name from the following result.

\begin{proposition}[{\cite{oneil}, \cite{clark-wc1}}]
    Let $k$ be a field.
    If $X\in\H^1(\Spec k,E)$ has $\per(X)=N$, then $\ind(X)=N$ if and only if for some
    lift $Y$ of $X$ to $\H^1(\Spec k,E[N])$, the period-index obstruction $Ob(Y)$ vanishes.
\end{proposition}

A result of Zarhin shows how the period-index obstruction map and the Weil pairing are related.

\begin{proposition}[{\cite{zarhin}}]\label{prop:zarhin}
    The Weil pairing in cohomology can be expressed via the period-index
    obstruction map $Ob_\Lscr$ as
  $$[Y,Z]_N=Ob_\Lscr(Y+Z)-Ob_\Lscr(Y)-Ob_\Lscr(Z)$$ 
    for $Y,Z\in\H^1(S,E[N])$.
\end{proposition}

\begin{proof}[Sketch of proof]
    By definition, the theta group $\Gscr_\Lscr$ is a central extension
    of $E[N]$ by $\Gm$ and as such there is a corresponding alternating
    pairing $E[N]\times E[N]\rightarrow\Gm$ associated to the central extension.
    Zarhin proves that the induced bilinear pairing
    $[-,-]_\Lscr\colon\H^1(S,E[N]) \times \H^1(S,E[N])$ on cohomology satisfies
    $$[Y,Z]_\Lscr=Ob_\Lscr(Y+Z)-Ob_\Lscr(Y)-Ob_\Lscr(Z).$$ It is just a
    statement about the connection between the boundary map in non-abelian
    cohomology and the pairing associated to the central extension.
    However, the pairing $[Y,Z]_\Lscr$ coincides with the Weil pairing $[Y,Z]_N$
    by Mumford~\cite[(5) on p.228]{mumford-abelian}.
\end{proof}

Recall that a function $f\colon A\rightarrow B$ of abelian groups is {\bf
quadratic} if $f(nx)=n^2f(x)$ for all integers $n$ and if the symmetric
function $b_f(x,y)=f(x+y)-f(x)-f(y)$ is bilinear. By
Proposition~\ref{prop:zarhin}, $b_{Ob_\Lscr}$ is bilinear. When $\Lscr$ is a
symmetric line bundle on $E$, meaning that $[-1]^*\Lscr\iso\Lscr$, then
$Ob_\Lscr$ is additionally quadratic with associated bilinear form
$b_{Ob_\Lscr}=[-,-]_N$
given by the cohomological Weil pairing by Proposition~\ref{prop:zarhin}.
To see this, Mumford uses an isomorphism $[-1]^*\Lscr\iso\Lscr$ to construct,
for any integer $n$, a
group endomorphism $\delta_n\colon\Gscr_\Lscr\rightarrow\Gscr_\Lscr$ fitting into a
commutative diagram
$$\xymatrix{
    1\ar[r]&\Gm\ar[r]\ar[d]^{n^2}&\Gscr_\Lscr\ar[r]\ar[d]^{\delta_n}&E[N]\ar[d]^n\ar[r]&0\\
    1\ar[r]&\Gm\ar[r]&\Gscr_\Lscr\ar[r]&E[N]\ar[r]&0
}$$
of exact sequences. This implies immediately the relation $Ob_\Lscr(nY)=n^2
Ob_\Lscr(Y)$ so that $Ob_\Lscr$ is quadratic.

\begin{example}
  The line bundle $\Oscr(N 0_E)$ is evidently symmetric, so that
  $Ob$ is a quadratic function with associated
  bilinear form given by the Weil pairing.
\end{example}

When $\Lscr$ is symmetric and the degree $N$ is odd, it follows that
we can compute $$Ob_\Lscr(Y)=\tfrac{1}{2}[Y,Y]_N$$ in $\H^2(S,\Gm)$.  Note
here that the alternating pairing $E[N]\times E[N]\rightarrow\Gm$
becomes symmetric in cohomology
$\H^1(S,E[N])\times\H^1(S,E[N])\rightarrow\H^2(S,\Gm)$.

Now, we specialize to the case where we have a symplectic full level $N$
structure, in a sense that we will make precise.  We consider the
standard symplectic bilinear form
$$(\bZ/N \times \mu_N) \times (\bZ/N \times \mu_N) \to \Gm$$ induced
from the evaluation pairing, which takes the form
$((a,\zeta),(b,\xi)) \mapsto \zeta^b \xi^{-a}$. A full level $N$
structure on an elliptic curve $E$ is an isomorphism of group schemes
$\varphi \colon E[N] \to \bZ/N \times \mu_N$;\footnote{Often, it is
  $\varphi^{-1}$ which is called the full level $N$ structure;
  moreover, when $N$ is not invertible, it is usually more desirable
  to use Drinfeld level $N$ structures to capture supersingular
  phenomena; see~\cite{katz-mazur}.} we say the full level $N$
structure $\varphi$ is symplectic if it is an isometry from the Weil
pairing on $E[N]$ to the standard symplectic form on
$\bZ/N\times\mu_N$.  

We remark that over a field, the moduli space of elliptic curves with a symplectic
full level $N$ structure is geometrically connected, and is a
component of the moduli space of elliptic curves with a full level $N$
structure.  For more remarks on this type of level structure, see
\cite[Section~4.1]{poonen_schaefer_stoll}.

Given any full level $N$ structure there is an induced map on
cohomology
$$
\H^1(\varphi) \colon \H^1(S, E[N]) \to \H^1(S, \bZ/N
\times \mu_N) = \H^1(S, \bZ/N) \times \H^1(S, \mu_N)
$$
so that to $Y \in \H^1(S, E[N])$ there is corresponding tuple
$(\chi, u) \in \H^1(S, \bZ/N) \times \H^1(S, \mu_N)$.  We now proceed
to show how this is related to the Brauer class $[\chi,u] \in \H^2(S, \Gm)$ of
the cyclic algebra $\Ascr(\chi,u)$.  First, we need the following.

\begin{lemma}\label{lem:cycliczarhin}
  The standard symplectic form on $\bZ/N\times\mu_N$ induces a pairing
  on cohomology
  $$[-,-]_\kappa\colon\H^1(S,\bZ/N\times\mu_N)\times\H^1(S,\bZ/N\times\mu_N)
  \to \H^2(S,\Gm)$$
  given by
    \begin{align*}
        [(\chi_0,u_0),(\chi_1,u_1)]_\kappa
        &=[\chi_0+\chi_1,u_0u_1]-[\chi_0,u_0]-[\chi_1,u_1]\\
        &=[\chi_0,u_1]+[\chi_1,u_0].
    \end{align*}
\end{lemma}

\begin{proof}
  After Zarhin~\cite{zarhin}, the important thing to check is that
  pulling back the central extension
  $1\rightarrow\Gm\rightarrow\GL_N\rightarrow\PGL_N\rightarrow 1$
  along the homomorphism $\bZ/N\times\mu_N\hookrightarrow\PGL_N$
  yields a central extension
  $1\rightarrow\Gm\rightarrow\Hscr_N\rightarrow\bZ/N\times\mu_N\rightarrow
  1,$ whose associated $2$-cocycle determines the standard symplectic
  form. This verification is routine: one lifts elements of
  $\bZ/N\times\mu_N$ to elements of $\GL_N$ (which we have already
  done in the previous section) and then one computes the
  commutator.\footnote{The group $\Hscr_N$ appearing here is a
    Heisenberg group.} For example, the commutator between the matrix
    $$
    \begin{pmatrix}
            0   &   0   &   \cdots  &   0   &   1\\
            1   &   0   &   \cdots  &   0   &   0\\
            0   &   1   &   \cdots  &   0   &   0\\
            \vdots  &   \vdots &   \ddots & \vdots    &   \vdots\\
            0   &   0   &   \cdots  &   1   &   0\\
        \end{pmatrix}
    $$
    and the diagonal matrix $(1,\zeta,\ldots,\zeta^{N-1})$ is the
    constant diagonal matrix $\zeta$ in the center.
\end{proof}

Now, we give our version of the obstruction theory of O'Neil and Clark.

\begin{proposition}\label{prop:oneil}
    Let $\Lscr$ be a symmetric line bundle of degree $N$ on $E$ and
  suppose that $\varphi \colon E[N] \to \bZ/N\times\mu_N$ is a
  symplectic full level $N$ structure and that
  $Y\in\H^1(S,E[N])$ corresponds to the tuple
    $\H^1(\varphi)(Y)=(\chi,u)\in\H^1(S,\bZ/N)\times\H^1(S,\mu_N)$.  If $N$ is odd, then
  $$Ob_\Lscr(Y)=\tfrac{1}{2}[Y,Y]_N=[\chi,u]$$
  in $\H^2(S,\Gm)$.  If $N$ is even, then
  $$Ob_\Lscr(Y)=[\chi,u]+[\chi,v]+[\sigma,u]$$ for classes $v\in\H^1(S,\mu_2)$ and
    $\sigma\in\H^1(S,\bZ/2)$ that depend only on $E$, $N$, $\varphi$,
    and $\Lscr$.
\end{proposition}

\begin{proof}
  The functorial assignment
  $Ob_\Lscr\colon\H^1(S,E[N])\rightarrow\H^2(S,\Gm)$ corresponds by Yoneda
  to a Brauer class on the classifying stack $\B (E[N])$.  A full level
  $N$ structure $\varphi\colon E[N] \to \bZ/N \times \mu_N$ induces an
  equivalence of stacks
  $\B(\varphi) \colon \B (E[N]) \to \B (\bZ/N \times \mu_N)$.  Applying
  $\B(\varphi)^{-1\ast}$ to the Brauer class on $\B(E[N])$
  corresponding to $Ob_\Lscr$, we obtain a class $\B(\varphi)^{-1\ast}(Ob_\Lscr)$
  in $\H^2(\B(\bZ/N \times \mu_N), \Gm)$, and we wish to view this
  class with respect to the decomposition~\eqref{eq:brB}.

  Given a Brauer class on $\B(\bZ/N\times\mu_N)$ determined via~\eqref{eq:brB}
    by an integer $n\pmod N$ and classes $v\in\H^1(S,\mu_N)$,
  $\sigma\in\H^1(S,\bZ/N)$, and $\alpha\in\H^2(S,\Gm)$, the associated
  functorial assignment $\H^1(S,\bZ/N \times \mu_N)\rightarrow\H^2(S,\Gm)$ is
\begin{equation}
\label{eq:assignment}
(\chi, u) \mapsto n[\chi,u] + [\chi,v] + [\sigma, u] + p^*\alpha
\end{equation}
where $p\colon \B (\bZ/N \times \mu_N) \to S$ is the structure morphism
and where we consider the cup products via the natural map
$\H^2(S,\mu_N) \to \H^2(S,\Gm)$.

We now consider the representation of $\B(\varphi)^{-1\ast}(Ob_\Lscr)$ with
respect to \eqref{eq:assignment}.  Since $Ob_\Lscr(0)=0$, we see that
$\alpha=0$. Since $\varphi$ is symplectic, the cohomological pairing
    induced by $\B(\varphi)^{-1\ast}(Ob_\Lscr)$, which is
    $$n[\chi_0,u_1]+n[\chi_1,u_0]$$ by Proposition~\ref{prop:zarhin}
    and~\eqref{eq:assignment} (since the linear part does not contribute), must agree with that induced by the
standard symplectic pairing on $\bZ/N\times\mu_N$, which is
    $$((\chi_0,u_0),(\chi_1,u_1))=[\chi_0+\chi_1,u_0+u_1]-[\chi_0,u_0]-[\chi_1,u_1]=[\chi_0,u_1]+[\chi_1,u_0]$$
    by Lemma~\ref{lem:cycliczarhin}.
    From this, we see that $n\equiv 1\pmod N$.
    Since $Ob_\Lscr$ is quadratic because $\Lscr$ is symmetric, we see that if $N$ is odd, then
$v=\sigma=0$. If $N$ is even, then we see that
$[2\chi,u^2]+[2\chi,v]+[\sigma,u^2]=4([\chi,u]+[\chi,v]+[\sigma,u])$,
which after cancellation gives $2([\chi,v]+[\sigma,u])=0$.  Since this
is true no matter the choice of $\chi$ and $u$, and even remains true after
    arbitrary field extensions, we see that $v$ and
$\sigma$ must be $2$-torsion, which completes the proof.
\end{proof}

The proof of the following refinement of Proposition~\ref{prop:oneil} is given
in~\cite{oneil:erratum}.

\begin{proposition}\label{prop:oneilerratum}
  Let $E$ be an elliptic curve with a symplectic full level $N$
  structure $\varphi\colon E[N]\rightarrow\bZ/N\times\mu_N$ and let
  $P = \varphi^{-1}(1,1)$. If $\Lscr = \Oscr(0_E + P + \dotsm + (N-1)P)$
  then the corresponding class $v \in \H^1(S,\mu_N)$ in
  Proposition~\ref{prop:oneil} vanishes. In particular, in the
  notation of that proposition,
    $$Ob_\Lscr(Y)=[\chi+\sigma,u]$$
  for some $\sigma \in \H^1(S,\bZ/N)$ that is trivial whenever $N$ is odd.
\end{proposition}

\begin{proof}
  Arguing as in Section~\ref{sec:cyclic}
  or~\cite[Sec.~4.3]{lieblich-surface}, we see that
  $\H^2(\B(\bZ/N),\Gm)\iso\H^1(S,\mu_N)\oplus\H^2(S,\Gm)$. This is
  compatible with the computation of the cohomology of
  $\B(\bZ/N\times\mu_N)$ of~\eqref{eq:brB}, which means that we can
  find $v$ by pullback of $\B(\varphi)^{-1}(Ob_\Lscr)$ from
  $\B(\bZ/N\times\mu_N)$ to $\B(\bZ/N)$.  The Brauer class
  corresponding to $v$ on $\B(\bZ/N)$ is that of the projective
  representation $\rho \colon \bZ/N\rightarrow\PGL_N^\Lscr$ induced by
  $\bZ/N\times
  1\xrightarrow{\varphi^{-1}}E[N]\hookrightarrow\PGL_N^\Lscr$.

  In general, computing the Brauer class on $\B(\bZ/N)$ from the
  projective representation is possible as follows.  Let
  $t_P \colon E \to E$ denote translation by $P$ and choose an
  isomorphism $t_P^*\Lscr\iso\Lscr$. Pulling back by $t_P$ gives an
  invertible map
  $\tau_P \colon p_\Lscr\rightarrow p_*t_P^*\Lscr\iso p_*\Lscr$ that
  lifts $\rho(P)$ to $\GL_N^\Lscr$ via
  $\GL_N^\Lscr\to\PGL_N^\Lscr$. As $\rho(P)$ has order $N$ in
  $\PGL_N^\Lscr$, the $N$th power of $\tau_P$ is a unit multiple of
  the identity $b\cdot I_N$ since the center of $\GL_N^\Lscr$ is still
  $\Gm$. The class of $b$ in $\H^1(S,\mu_N)$ is then the Brauer class
  on $\B(\bZ/N)$ corresponding to $\rho$.  To see this, note that
  $\rho$ lifts to a representation $\bZ/N\rightarrow\GL_N^\Lscr$ if
  and only if we can rescale $\tau_P$ by a unit to obtain $\tau'_P$
  such that $(\tau'_P)^N=I_N$. If $b=c^N$, then $\tau'_P=c^{-1}\tau_P$
  works, and conversely.

  We have to show that in the setup of the proposition, this $b$ is
  itself an $N$th power,
  cf.~\cite[Corollary~2.8]{oneil-jacobians}. The line bundle $\Lscr$
  is the divisor class of the divisor $0_E+P+\cdots+(N-1)P$. This
  divisor is in fact invariant by translation by $P$ so there is a
  preferred identification of $\Lscr$ and $t_P^\ast\Lscr$.
  Specifically,
  $\Lscr\we\Oscr(0_E)\otimes\Oscr(P)\otimes\cdots\otimes\Oscr((N-1)P)$
  and $t_P^\ast\Lscr$ is a cyclic permutation. In particular, we can
  arrange for $\tau_P$ to act on
  $$p_*\Lscr\iso p_*\Oscr(0_E)\oplus p_*\Oscr(P)\oplus\cdots\oplus p_*\Oscr((N-1)P)$$
  via a cyclic permutation (we use here that the genus of $E$ is $1$
  over $S$), which has associated unit element $b=1$. This completes
  the proof.
\end{proof}

When $E$ has full rational $N$-torsion over a field $k$ (necessarily
containing a primitive $N$th root of unity),
O'Neil~\cite[Prop.~3.4]{oneil} states the first part of
Proposition~\ref{prop:oneil}, but for any $N$, and then provides a
correction in~\cite{oneil:erratum} which amounts to
Proposition~\ref{prop:oneilerratum}.  Around the same time, again
assuming that $E$ has full rational $N$-torsion,
Clark~\cite[Sec.~3.2]{clark-wc1} states the second part of the result,
but for all $N$, while pointing out the issue with $N=2$ that was
uncovered by a referee, which implies that in general $v$ and $\sigma$
are non-zero for the obstruction map $Ob$ corresponding to
$\Oscr(N0_E)$. Our version of the result, while holding over a general
base $S$, utilizes our more general notion of full level $N$
structure, which in particular does not presuppose the presence of
$N$th roots of unity.

\begin{remark}
  Using oriented automorphisms of $p_*\Lscr$, one can obtain
  an oriented theta group, which is an extension of $E[N]$ by
  $\mu_N$. Then, the definition of the Weil pairing, the
  period-obstruction map, and the above lemma carry over with values
  in $\H^2(S,\mu_N)$ as well.  Specifically, we can pull back the
  extension
    $$1\rightarrow\mu_N\rightarrow\SL_N^\Lscr\rightarrow\PGL_N^\Lscr\rightarrow
    1$$ along
    $E[N]\rightarrow\PGL_N^\Lscr$. Propositions~\ref{prop:zarhin},
    \ref{prop:oneil}, and~\ref{prop:oneilerratum} now hold with target in $\H^2(S,\mu_N)$
    instead of $\H^2(S,\Gm)$ with the exception that in the proofs of
    Propositions~\ref{prop:oneil} and~\ref{prop:oneilerratum}, in order to conclude that the
    contribution of the constant class classes $\H^2(S,\mu_N)$ to $Ob_\Lscr$
    vanishes, we need to know that $c_1(p_*\Lscr)=0$, which is always
    true locally on $S$. The details are left to the reader.
\end{remark}

\begin{warning}
    We will use Proposition~\ref{prop:oneilerratum} in the proof of
    Theorem~\ref{thm:c}, but note that we cannot improve the statement to split
    $\mu_N$-gerbes, despite the previous remark. This is because a later
    result, Lemma~\ref{lem:pio}, does not apply to $\mu_N$-cohomology classes,
    as Example~\ref{ex:bigprimes} shows.
\end{warning}

\section{Proofs}\label{sec:proofs}

\subsection{Splitting $\mu_N$-gerbes}
\label{subsec:splitting-mu_N}

As explained in the introduction, we are interested in the following
general question, which motivates our approach to the problem of splitting
Brauer classes by genus 1 curves.

We first recall that, for any proper geometrically connected scheme
$X$ over a field $k$, a consequence of the long exact sequence of
cohomology associated to the Kummer sequence are the exact sequences:
$$
0 \to k^\times/(k^\times)^N \to \H^1(X,\mu_N) \to \Pic(X)[N] \to 0,
$$
$$
0 \to \Pic(X)/N\Pic(X) \to \H^2(X,\mu_N) \to \H^2(X,\Gm)[N] \to 0.
$$
When $X$ is projective, the right-most term in the second statement is
$\Br(X)[N]$ via Gabber's
$\Br=\Br'$ result; see~\cite{dejong}.
In particular, we can view classes in $\H^2(X,\mu_N)$ as lifts of
$N$-torsion Brauer classes.

\begin{question}
  Let $k$ be a field and $\beta \in \H^2(\Spec k,\mu_N)$. Is
  there a smooth proper geometrically connected $k$-scheme $X$ such
  that $\beta$ pulls back to zero in $\H^2(X,\mu_N)$?  
\end{question}

A class $\beta \in \H^2(\Spec k,\mu_N)$ can be considered as the cohomology
class associated to a $\mu_N$-gerbe, hence if it pulls back to zero in
$\H^2(X,\mu_N)$ then we say that $X$ splits the $\mu_N$-gerbe $\beta$.
If we drop the hypothesis that $X$ is geometrically connected, we can
of course just take $X$ to be the spectrum of a finite field extension
that splits $\beta$.  One might be tempted to take $X$ to be a
Severi--Brauer variety associated to $\beta$, but
Lemma~\ref{lem:never_split} below shows that these varieties do not split
$\mu_N$-gerbes.

To think about this question, consider the Leray--Serre spectral
sequence 
$$
\E_2^{s,t}=\H^s(\Spec k,\R p_*^t\mu_n)\Longrightarrow\H^{s+t}(X,\mu_n)
$$ 
for the sheaf $\mu_N$ associated to the structure morphism
$p\colon X \to \Spec k$.  We have $\R p_*^0\mu_n\iso\mu_n$ since $X$ is
geometrically connected and $\R p_*^1\mu_n\iso\Pic_{X/k}[n]$ since $X$
is proper over $k$.  Hence the exact sequence of low-degree terms, together
with the exact sequence in degree 1 recalled above, reads as
\begin{equation}
\label{eq:Leray_mu_N}
0 \to \Pic(X)[N] \to \Pic_{X/k}[N](k) \xto{\d_2^{0,1}} \H^2(\Spec k,\mu_N) \to \H^2(X,\mu_N)
\end{equation}
where $\d_2^{0,1}$ is the appropriate differential and the final map
is the pull back in cohomology.  This exact sequence shows that
$X$ splits the $\mu_N$-gerbe $\beta$ if and only if $\beta$ arises as
the differential of an $N$-torsion point on the Picard scheme
$\Pic_{X/k}$.  In particular, we have the following.

\begin{lemma}
\label{lem:never_split}
  Let $X$ be smooth proper geometrically connected $k$-scheme. The
  pullback map $\H^2(\Spec k,\mu_n)\rightarrow\H^2(X,\mu_n)$ is
  injective, in other words, $X$ does not split any nontrivial
  $\mu_N$-gerbe over $k$, if and only if the differential
  $d_2^{0,1}\colon \Pic_{X/k}[N](k) \to \H^2(k,\mu_N)$ vanishes.
\end{lemma}

As a special case, the lemma shows that if $\Pic_{X/k}[N](k)=0$, then
$\H^2(\Spec k,\mu_N)\rightarrow\H^2(X,\mu_N)$ is injective and hence
$X$ does not split any $\mu_N$-gerbe.  This condition holds, in
particular, whenever the Picard group of $X$ is geometrically
torsion-free, e.g., for K3 surfaces, Severi--Brauer varieties, and for
smooth complete intersections of dimension at least $2$.

Another special case is that of Enriques surfaces $X$: if $X$ is
classical (e.g., if the characteristic of $k$ is not $2$), then
$\Pic(X)[2]\iso\Pic_{X/k}[2](k)\iso\bZ/2$ is generated by the
canonical line bundle; if $X$ is singular or supersingular in characteristic
$2$, then
$\Pic_{X/k}[2]\iso\mu_2$ or $\alpha_2$, respectively, so the group of
$k$-rational points is $0$ (see~\cite{bombieri-mumford-2}). In either
case, $\d_2^{0,1}=0$ so Enriques surfaces cannot split $\mu_2$-gerbes.

We will now give some examples where genus $1$ curves do split
$\mu_N$-gerbes and others where they split Brauer classes but not
their associated $\mu_N$-gerbes.  More subtly, we will give examples
where a genus $1$ curve $X$ does not split any $\mu_N$-gerbe lifting
$\alpha$ of period $N$, but that $X$ does split some $\mu_{NM}$-gerbe
lifting $\alpha$ for $M > 1$.  Many of these
examples rely on the following, cf.\
\cite[Theorem~2.1]{ciperiani-krashen}.

\begin{lemma}
\label{lem:splitting}
Let $X$ be a genus $1$ curve over a field $k$ with Jacobian
$\Pic_{X/k}^0$. There is a commutative diagram
    $$\xymatrix@R=16pt@C=16pt{
        &0\ar[d]&0\ar[d]&\\
        0\ar[r]&\Pic^0(X)\ar[r]\ar[d]&\Pic_{X/k}^0(k)\ar[r]\ar[d]&\Br(k)\ar[d]\\
        0\ar[r]&\Pic(X)\ar[r]\ar[d]&\Pic_{X/k}(k)\ar[r]\ar[d]&\Br(k)\\
        &\bZ\cdot\ind(X)\ar[d]\ar[r]&\bZ\cdot\per(X).\ar[d]&\\
        &0&0
    }$$
    In particular, letting $\Br^0(X/k)=\im(\Pic_{X/k}^0(k)\rightarrow\Br(k))$ and
    $\Br(X/k)=\im(\Pic_{X/k}(k)\rightarrow\Br(k))=\ker(\Br(k)\rightarrow\Br(X))$,
    there is an exact sequence
    \begin{equation}\label{eq:perind}
    0\rightarrow\Br^0(X/k)\rightarrow\Br(X/k)\rightarrow\bZ/(\ind(X)/\per(X))\rightarrow 0.
    \end{equation}
\end{lemma}

\begin{corollary}
\label{ex:rank0}
Let $E$ be an elliptic curve over a field $k$ such that $E(k)$ is an
$N$-torsion group. Let $X$ be an $E$-torsor such that
$\per(X)=\ind(X)$.  If $X$ splits a Brauer class
$\alpha \in \Br(k)[N]$, then it splits the unique lift
$\beta \in \H^2(\Spec k,\mu_N)$ of $\alpha$.
\end{corollary}

\begin{proof}
  Indeed, by Lemma~\ref{lem:splitting} we see that
  $\alpha\in\Br^0(X/k)$, hence $\alpha = \d^{0,1}_2(P)$ for some point
  $P \in E(k) = \Pic_{X/k}^0(k)$.  However, since $E(k)$ is an
  $N$-torsion group, $P \in \Pic_{X/k}[N](k)$.  Compatibility of the
  Leray--Serre spectral sequences for $\mu_{N}$ and $\Gm$-cohomology
  provides a commutative diagram
    $$\xymatrix@R=16pt{
      \Pic_{X/k}[N](k) \ar[r]^(.45){\d_2^{0,1}}\ar[d]&\H^2(\Spec k,\mu_{N})\ar[dd]\\
      \Pic_{X/k}^0(k)  \ar[d]&\\
      \Pic_{X/k}(k) \ar[r]^(.45){\d_2^{0,1}}&\H^2(\Spec k,\Gm) }$$
    showing that $\beta$, considered as an element of
    $\H^2(\Spec k, \mu_{N})$, is in the image of a class from
    $\Pic_{X/k}[N](k)$, so that $X$ splits the $\mu_{N}$-gerbe
    $\beta$.  
\end{proof}

\begin{example}
If $k$ has characteristic not 2, then the quaternion algebra $[a,b]$ is split by
the genus $1$ curve $X$ defined by
$$y^2=ax^4+b.$$ 
Geometrically, one can see this since $(x,y) \mapsto (x^2,y)$
determines a (finite degree 2) morphism from $X$ to the conic
$y^2=ax^2+b$ determined by $[a,b]$.  The relative Brauer group
$\Br(X/k)$ was studied by \cite{han} and
\cite[Sec.~5.2.1]{ciperiani-krashen} in terms of the Jacobian $E$ of
$X$, which is defined by $y^2=x^3-4abx$, giving examples of the
phenomenon in Corollary~\ref{ex:rank0}.  Here, when $[a,b]$ is
nonsplit, one has $\per(X)=\ind(X)=2$ and in general one can find
examples where $E(k)$ is $2$-torsion. For example, this is the case
for $a=2$ and $b=3$ where the Jacobian is given up to isomorphism by
$y^2=x^3-24x$, the curve \texttt{2304.n1} from
\texttt{lmfdb.org}~\cite{lmfdb}, which has $E(\bQ)\iso\bZ/2$.  Thus,
the $\mu_2$-gerbe $2\cup 3\in\H^2(\Spec\bQ,\mu_2)$ is split by the genus
$1$ curve that is the unique smooth
proper model of $y^2=2x^4+3$.
\end{example}

\begin{example}
  An analysis similar to the proof of Corollary~\ref{ex:rank0} shows
  that if $X$ is an $E$-torsor with $E(k)$ a torsion group, and the
  relative Brauer group $\Br(X/k)$ is not cyclic, then there must be
  classes of $\H^2(\Spec k,\mu_N)$ for some $N$ split by $X$. An
  explicit example is given in~\cite[Sec~5.2.3]{ciperiani-krashen}.
\end{example}

\begin{example}\label{ex:mnotn}
  We can also examine the proof of Corollary~\ref{ex:rank0} to find
  more subtle phenomena.  If $E$ is an elliptic curve with
  $E(k)\iso\bZ/NM$ generated by a point $P$, where $N,M>1$, and $X$ is and $E$-torsor with
  $\d^{0,1}_2(P) = \alpha\in \Br(k)$ having period $N$, then we see that
  $X$ does not split the $\mu_N$-gerbe associated to $\alpha$ though
  it does split the associated $\mu_{NM}$-gerbe, hence also $\alpha$.
  An explicit example can be taken from our proof of
  Theorem~\ref{thm:a} in Section~\ref{subsec:proof_A}.  Let $E$ be the
  elliptic curve $$y^2 + xy + 4y = x^3 + 4x^2,$$ which is
  \texttt{130.b4} from \texttt{lmfdb.org}~\cite{lmfdb} and satisfies
  $E(\bQ)\iso\bZ/4$.  For any $\chi \in \H^1(\Spec \bQ,\bZ/4)$ we
  define an $E$-torsor $X_\chi$ in Section~\ref{subsec:mu_n-isogeny}
  such that $\d^{0,1}_2(P) = \chi\cup 2^6 = 2(\chi\cup 2)$, where $\chi\cup 2$
  is the quartic cyclic symbol.  Hence the image of
    $\d^{0,1}_2 : E[4](\bQ) \to \H^2(\Spec\bQ,\mu_4)$ is at most a subgroup of
  order 2, which is nontrivial if and only if the cyclic algebra
  $[\chi,2]$ has period 4.  In this case, $X_\chi$ splits the
    $\mu_4$-gerbe associated to the Brauer class $\alpha = 2[\chi,2]$,
  but not the associated $\mu_2$-gerbe.
\end{example}

\begin{example}\label{ex:bigprimes}
  Let $k$ be a global field and let $N$ be invertible in $k$. Fix a
  positive integer $M$ dividing $N$. Sharif showed in~\cite{sharif}
  that for any elliptic curve $E$ over $k$, there exists an $E$-torsor
  $X$ with $\per(X)=N$ and $\ind(X)=MN$. (Clark and
  Sharif~\cite{clark-sharif} had previously shown this result when
  $M=N$.) If $M>1$, then it follows from the exact
  sequence~\eqref{eq:perind} that $X$ splits some non-zero Brauer classes,
    specifically the period-index obstruction classes $Ob(Y)$ where $Y$ ranges
    over the lifts of $X$ to $\H^1(\Spec k,E[N])$. If $E$ is chosen so
  that $E[N](k)=0$, then we see that $X$ cannot split any
  $\mu_{N}$-gerbe. This gives examples of the phenomenon of a genus
  $1$ curve splitting a Brauer class $\alpha$ but not any associated
  $\mu_N$-gerbe $\beta$.  In particular, over $\bQ$, if $N\geq 11$ is
  a prime, then $E[N](\bQ)=0$ for any elliptic curve by Mazur's
  theorem~\cite{mazur} and there is some Brauer class $\alpha$ of
  order $N$ split by an $E$-torsor. In fact, there are infinitely many
  such by~\cite[Thm.~2]{clark-sharif}. This establishes the existence
  of Brauer classes $\alpha$ split by a genus $1$ curve where the
  associated $\mu_N$-gerbe is not split by any genus $1$ curve.
\end{example}

Finally, we examine the case of splitting $\mu_N$-gerbes over local fields.
Here is the non-archimedean case.

\begin{proposition}\label{prop:local}
  If $K$ is a non-archimedean local field and $N\geq 2$, then every $\mu_N$-gerbe
  $\beta\in\H^2(\Spec K,\mu_N)\iso\bZ/N$ is split by a genus $1$
  curve.
\end{proposition}

\begin{proof}
    It suffices to consider the case when $\beta$ has exact order $N$ in
    $\H^2(\Spec K,\mu_N)$. Indeed, every exact order $N$ element of $\H^2(\Spec
    K,\mu_{MN})$ is in the image of the natural map $\H^2(\Spec
    K,\mu_N)\rightarrow\H^2(\Spec K,\mu_{MN})$ induced by
    $\mu_N\hookrightarrow\mu_{MN}$. So, assume that $\beta$ has exact order $N$
    in $\H^2(\Spec K,\mu_N)$. There is an elliptic curve $E$ over $K$ with an exact order $N$ point $P$.
    Indeed, as explained to us by Pete Clark, this follows from the existence
    of Tate curves; see~\cite{tate}. Let $E$ be the Tate elliptic curve with $E(K)\iso
    K^\times/(\varpi^N)^\bZ$, where $\varpi$ is a uniformizer of $K$. As an
    element of $K^\times$, $\varpi$ reduces to give an exact order $N$ point of $E(K)$.
    Let $\alpha$ be the image of $\beta$ in $\Br(K)$.
    By the non-degeneracy of the Tate pairing, there is an $E$-torsor $X$ such
    that $[P,X]_T=\alpha$. It follows that $X$ splits $\alpha$. Specifically,
    viewing $P$ as a $K$-point of the Jacobian $\Pic_{X/K}$, we see that
    $\d_2^{0,1}(P)=\alpha$ where the $\d_2$-differential comes from the
    Leray--Sere spectral sequence for $\Gm$-cohomology. Since $P$ is
    $N$-torsion, we also have $\d_2^{0,1}(P)=\beta$, where we view $P$ as a
    section of $\Pic_{X/K}[N]$ and work in the Leray--Serre spectral sequence
    for $\mu_N$-cohomology. In particular, $X$ splits the $\mu_N$-gerbe
    $\beta$.
\end{proof}

Now, we consider the real case. To begin, we recall a well-known
result about Legendre curves and a classification of the
$2$-torsion of real elliptic curves.

\begin{lemma}
\label{lem:legendre_twist}
An elliptic curve $E$ over a field $k$ of characteristic not $2$ has
full rational $2$-torsion if and only if $E$ is a quadratic twist of a
Legendre curve.
\end{lemma}

\begin{proof}
  Given a Weierstrass equation $y^2=f(x)$ of $E$, the rational
  $2$-torsion points correspond to roots of $f(x)$.  Hence $E$ has
  full rational $2$-torsion if and only if it admits a Weierstrass
  equation of the form $y^2 = (x-e_1)(x-e_2)(x-e_3)$ for some distinct
  $e_1, e_2, e_3 \in k$. In this case, the change of variables 
  $(x,y) \mapsto \bigl( (e_2-e_1)x + e_1, (e_2-e_1)^2y \bigr)$ transforms
  the Weierstrass equation into
  $(e_2-e_1) y^2 = x(x-1)(x-\frac{e_3-e_1}{e_2-e_1})$, which is the
  quadratic twist of a Legendre curve.
  Conversely, any Legendre curve has full rational $2$-torsion, hence
  any quadratic twist does as well, since negation acts trivially on
  the group scheme $E[2]$.
\end{proof}

\begin{remark}
  The proof of Lemma~\ref{lem:legendre_twist} shows that if $E$ has
  full rational $2$-torsion, then it is isomorphic to a Legendre curve
  if and only if any of the elements $e_j-e_i \in k^\times$ for
  $i \neq j$ is a square.  In particular, over the field of real
  numbers, every elliptic curve with full rational $2$-torsion is
  isomorphic to a Legendre curve.
\end{remark}

\begin{lemma}
\label{lem:1728}
Let $E$ be an elliptic curve over $\bR$ with $j$-invariant $j(E)$.
\begin{enumerate} 
    \item[{\rm (1)}] If $E$ has full rational $2$ torsion then $j(E) \geq 1728$.  In
  this case, there is a unique nontrivial $E$-torsor $X$, which has
  index 2.

\item[{\rm (2)}] If $E$ does not have full rational $2$ torsion, then $j(E) \leq
  1728$.  In this case, $E[2]$ is isomorphic to the Weil restriction
  group scheme $R_{\bC/\bR}\bZ/2$ and all $E$-torsors are split.

\item[{\rm (3)}] If $j(E)=1728$ then $E$ can have full rational $2$ torsion or
  not, depending on the sign in the Weierstrass equation $y^2=x^3\mp x$ of $E$.
\end{enumerate}
\end{lemma}

\begin{proof}
  The $2$-torsion group scheme $E[2]$ is an extension
  $0\rightarrow\bZ/2\rightarrow E[2]\rightarrow\bZ/2\rightarrow 0$.
  Indeed, every elliptic curve over $\bR$ admits a point of order $2$
  since any cubic polynomial over $\bR$ has a root. Since the extension
  is geometrically split, either $E[2] \iso \bZ/2 \times \bZ/2$ when
  $E$ has full rational $2$-torsion or $E[2] \iso E_{\bC/\bR}\bZ/2$
  when $E[2]$ does not have full rational $2$-torsion.  Now consider
  the exact sequence
$$
0 \to E(\bR)/2E(\bR) \to \H^1(\Spec \bR, E[2]) \to \H^1(\Spec \bR,
E)[2] \to 0
$$
in cohomology.

If $E$ has full rational $2$-torsion, then $E(\bR)$ has two connected
components.  Since the identity component is $2$-divisible,
$E(\bR)/2E(\bR) \iso \bZ/2$ is isomorphic to the component group of
$E(\bR)$.  Since
$$
\H^1(\Spec \bR,E[2]) \iso \H^1(\Spec \bR,\bZ/2 \times \bZ/2) \iso
\bZ/2 \times \bZ/2,
$$
the above exact sequence implies that
$\H^1(\Spec \bR,E)[2] \iso \bZ/2$.  However, since any $E$-torsor has
index, and hence period, dividing $2$, we deduce that
$\H^1(\Spec \bR,E) \iso \bZ/2$. Thus there is a unique nontrivial
$E$-torsor, which has index 2.

If $E$ does not have full rational $2$-torsion, then
$$
\H^1(\Spec \bR, E[2]) = \H^1(\Spec \bR, R_{\bC/\bR}\bZ/2) = \H^1(\Spec
\bC, \bZ/2) = 0,
$$
which then implies that $\H^1(\Spec \bR,E)[2] = 0$.  Again, since
every $E$-torsor has index, hence period, dividing 2, we have that
$\H^1(\Spec \bR,E) = 0$.  Hence all $E$-torsors are split.

We proceed to prove the statements about the $j$-invariant.  By
Lemma~\ref{lem:legendre_twist}, $E$ has full rational $2$-torsion if
and only if $E$ is a quadratic twist of a Legendre curve
$y^2 = x(x-1)(x-\lambda)$ for some $\lambda \neq 0,1$. Since twisting
preserves the $j$-invariant, if $E$ has full rational $2$-torsion then
$j(E)=2^8\tfrac{(\lambda^2-\lambda+1)^3}{\lambda^2(\lambda-1)^2}$.
One can compute that $j(E) \geq 1728$ and that the minimum value of
$1728$ is obtained, for example when $\lambda = 1$. Conversely, by
continuity, every number $j>1728$ is obtained as the $j$-invariant of
a Legendre curve, thus every elliptic curve with $j$-invariant
$j(E)>1728$ is a quadratic twist of a Legendre curve, hence has full
rational $2$-torsion.

Finally, we remark that the unique elliptic curve $E$ over $\bR$ with
$j(E) = 1728$ and full rational $2$-torsion has Weierstrass equation
$y^2=x^3-x$, which is \texttt{32.a3} in the
\texttt{lmfdb.org}~\cite{lmfdb}.  The only other elliptic curve $E$
over $\bR$ with $j(E) = 1728$, \texttt{64.a4} in the
\texttt{lmfdb.org}~\cite{lmfdb} with Weierstrass equation
$y^2 = x^3+x$, does not have full rational $2$-torsion.  These two
elliptic curves are quartic twists of each other,
cf.~\cite[Corollary~X.5.4.1]{silverman}.
\end{proof}

Finally, we classify when a genus $1$ curves can split the (unique)
$\mu_2$-gerbe over $\bR$.

\begin{proposition}
  Let $\beta\in\H^2(\Spec\bR,\mu_2)\iso\bZ/2$ be the $\mu_2$-gerbe
  corresponding to the unique nontrivial element of $\Br(\bR)$. Let
  $E$ be an elliptic curve over $\bR$ with $j$-invariant $j(E)$. Then
  $\beta$ is split by an $E$-torsor if and only if $E$ has full
  rational $2$-torsion, or equivalently if and only if $j(E) > 1728$ or $E$ is the
  unique elliptic curve with $j(E)=1728$ and full rational
  $2$-torsion.
\end{proposition}

\begin{proof}
  By Lemma~\ref{lem:1728}, we must show that if $E$ has full rational
  $2$-torsion, then the unique nontrivial $E$-torsor $X$ splits
  $\beta$.  Choosing a (necessarily symplectic) full level $2$
  structure on $E$, we get an obstruction map of the
  form $$Ob(Y)=[\chi,u]+[\chi,v]+[u,\sigma]$$ for
  $Y\in\H^1(\Spec\bR,E[2])$, where $\chi,u,v,\sigma$ are as in
  Proposition~\ref{prop:zarhin}. Whatever the nature of $v,\sigma$,
  some choice of $\chi,u$ (corresponding to some choice of
  $E[2]$-torsor $Y$ lifting $X$) will satisfy $Ob(Y)=\alpha$, where
  $\alpha$ is the Brauer class of $\beta$ (corresponding to Hamilton's
  quaternions).  By Lemma~\ref{lem:pio}, $X$ splits $\alpha$.  Since
  $\per(X)=\ind(X)=2$, the exact sequence~\eqref{eq:perind} implies
  that the exact sequence of low degree terms of the Leray spectral
  sequence for $\Gm$ on $X$ reads as
  $$0\rightarrow\Pic(X)\rightarrow
  E(\bR)\rightarrow\H^2(\Spec\bR,\Gm) \rightarrow 0.$$ This implies
  that $\Pic(X)$ is isomorphic to the connected component of the
  identity of $E(\bR)$.  In particular, we have
  $\Pic(X)[2]\iso\bZ/2$. Since $E(\bR)[2])\iso\bZ/2\times\bZ/2$, the
  exact sequence \eqref{eq:Leray_mu_N} shows that there must be a
  non-trivial differential
  $\Pic_{X/\bR}[2](\bR)\rightarrow\H^2(\Spec\bR,\mu_2)$. In other
  words, $X$ splits $\beta$, as desired.
\end{proof}

\subsection{Descent via $\mu_N$-isogenies}
\label{subsec:mu_n-isogeny}

In this section, we describe the specific tools we
will use to prove Theorems~\ref{thm:a} and~\ref{thm:b}. Let
$$0\rightarrow\mu_N\rightarrow E\xrightarrow{\varphi} E'\rightarrow
0$$ be a $\mu_N$-isogeny with dual
isogeny
$$0\rightarrow\bZ/N\rightarrow
E'\xrightarrow{\widehat{\varphi}}E\rightarrow 0.$$ Implicitly, in
identifying $\ker(\widehat{\varphi}) \iso \bZ/N$ we have made a choice
of an $S$-point $P$ of $E'$ of order $N$. We have a commutative diagram
\begin{equation}
\label{eq:liftpoint}
\xymatrix{
  0\ar[r]&\mu_N\ar@{=}[d]\ar[r]&E[N]\ar[r]\ar[d]&\bZ/N\ar[d]\ar[r]&0\\
  0\ar[r]&\mu_N\ar[r]&E\ar[r]&E'\ar[r]&0 }
\end{equation}
with exact rows.  Taking
boundaries we obtain a commutative diagram
$$\xymatrix{
    \H^1(S,\bZ/N)\ar[r]^\delta\ar[d]&\H^2(S,\mu_N)\ar@{=}[d]\\
    \H^1(S,E')\ar[r]^\delta&\H^2(S,\mu_N).
}$$
The crux of our argument is the following simple lemma.

\begin{lemma}
\label{lem:split_delta}
    Let $\chi\in\H^1(S,\bZ/N)$ have associated $E'$-torsor $X_\chi$. Then,
    $X_\chi$ splits $\delta(\chi)$.
\end{lemma}

\begin{proof}
    Indeed, $X_\chi$ splits the class $X_\chi$ itself and hence it splits
    $\delta(\chi)=\delta(X_\chi)$.
\end{proof}

It remains to compute the boundary map
$\delta\colon\H^1(S,\bZ/N)\rightarrow\H^2(S,\mu_N)$.
We can write the cohomological Weil pairing associated to the isogeny
$\varphi$ as a pairing
$$(-,-)_\varphi\colon\H^1(S,\mu_N)\times\H^1(S,\bZ/N)\rightarrow\H^2(S,\mu_N)$$
as in Definition~\ref{def:weil}.
The obstruction to lifting $P$ to
an $S$-point of $E[N]$, via the coboundary map
$\delta\colon\H^0(S,\bZ/N)\rightarrow\H^1(S,\mu_N)$ from the exact sequence \eqref{eq:liftpoint}, gives
a distinguished element $\delta(P)\in\H^1(S,\mu_N)$.

\begin{lemma}
\label{lem:delta_symbol}
    We have $\delta(\chi)=\delta(P)\cup\chi=(\delta(P),\chi)_\varphi$ in $\H^2(S,\mu_N)$.
\end{lemma}

This is a special case of the following, more general lemma.

\begin{lemma}\label{lem:coboundary}
  Let $S$ be a scheme and let $A$ be an $N$-torsion commutative group
  scheme over $S$ which sits in an
  extension
  $$0\rightarrow\mu_N\rightarrow A\rightarrow\bZ/N\rightarrow 0$$ of
  fppf sheaves of abelian groups. Let $\epsilon\in\H^1(S,\mu_N)$ be the image
  of $1$ via the coboundary map
  $(\bZ/N)(S)\rightarrow\H^1(S,\mu_N)$. Then, for any $i$, the map
  $\H^i(S,\bZ/N)\rightarrow\H^{i+1}(S,\mu_N)$ is given by the cup
  product
  map $$\epsilon\cup(-)\colon\H^i(S,\bZ/N)\rightarrow\H^{i+1}(S,\mu_N).$$
\end{lemma}

\begin{proof}
    Since $A$ is $N$-torsion,
    the extension $0\rightarrow\mu_N\rightarrow A\rightarrow\bZ/N\rightarrow 0$
    is determined by a homotopy class of maps
    $\bZ/N\rightarrow\mu_N[1]$ in the derived category of fppf sheaves of
    $\bZ/N$-modules. Since the constant sheaf $\bZ/N$ is the unit object of this category, we have
    $$\Hom_S(\bZ/N,\mu_N[1])\iso\H^1(S,\mu_N).$$ Now, giving a class
    $\chi\in\H^i(S,\bZ/N)$ is equivalent to giving a homotopy class of maps
    $\bZ/N\rightarrow\bZ/N[i]$ in the derived category.
    We can compose with the suspension $$\epsilon[i]\colon\bZ/N[i]\rightarrow\mu_N[i+1]$$ to obtain
    the Yoneda/cup product $$\bZ/N\xrightarrow{\chi}\bZ/N[i]\xrightarrow{u}\mu_N[i+1],$$
    which is what we wanted to show.
\end{proof}

\subsection{Splitting $\mu_N$-gerbes with abelian scheme
torsors}\label{sec:abelian}

Recall the following theorem of
Raynaud~\cite[Thm.~3.1.1]{berthelot-breen-messing}.

\begin{theorem}[Raynaud]
    Let $A\rightarrow S$ be a finite flat $N$-torsion commutative group scheme. Zariski
    locally on $S$ there exists an abelian scheme $J\rightarrow S$ and an
    inclusion $A\rightarrow J[N]$.
\end{theorem}

We use Raynaud's theorem to show that every $\mu_N$-gerbe over a field is split
by a torsor for an abelian variety.

\begin{theorem}\label{thm:ms}
    Let $k$ be a field.
    If $\beta\in\H^2(\Spec k,\mu_N)\iso\Br(k)[N]$, then there exists an abelian
    scheme $J'\rightarrow \Spec k$ and a $J'$-torsor
    $X\rightarrow \Spec k$ such that $\beta$ restricts to $0$ in $\H^2(X,\mu_N)$.
\end{theorem}

\begin{proof}
    Let $K$ be a finite \'etale extension of $k$ which splits $\beta$.
    Denote by $p\colon\Spec K\rightarrow\Spec k$ the
    morphism of schemes and let $A=p_*\mu_N$, which is a finite flat group
    scheme on $\Spec k$. Geometrically, $A$ is a direct sum of $r$ copies of
    $\mu_N$ where $r$ is the degree of $K$ over $k$. There is a natural
    injective unit
    map $\mu_N\rightarrow A=p_*\mu_N$ and we let $A'$ be the quotient. The Leray spectral
    sequence $\E_2^{s,t}=\H^s(\Spec k,\R p_*^t\mu_N)\Rightarrow\H^{s+t}(\Spec
    K,\mu_N)$ gives rise to an exact sequence of low degree terms
    $$\H^1(\Spec k,\mu_N)\rightarrow\H^1(\Spec
    K,\mu_N)\rightarrow\H^0(\Spec k,\R p_*^1\mu_N)\rightarrow\H^2(\Spec
    k,\mu_N)\rightarrow\H^2(\Spec K,\mu_N),$$
    which we see we can rewrite as part of the long exact sequence
    $$\H^1(\Spec k,\mu_N)\rightarrow\H^1(\Spec
    k,A)\rightarrow\H^1(\Spec k,A')\xrightarrow{\delta}\H^2(\Spec
    k,\mu_N)\rightarrow\H^2(\Spec k,A)$$
    in cohomology arising from the short exact sequence
    $0\rightarrow\mu_N\rightarrow A\rightarrow A'\rightarrow 0$.
    Since $\beta$ maps to $0$ in $\H^2(\Spec k,A)\iso\H^2(\Spec K,\mu_N)$,
    there is some class $\sigma\in\H^1(\Spec k,A')$ such that
    $\delta(\sigma)=\beta$. Using Raynaud's theorem, fix an abelian
    scheme $J$ together with an embedding $A\hookrightarrow J$. Let
    $J'=J/\mu_N$ where $\mu_N\subseteq A\subseteq J$. In this way, we obtain a
    commutative  diagram
    $$\xymatrix{
        0\ar[r]&\mu_N\ar[r]\ar[d]&A\ar[r]\ar[d]&A'\ar[r]\ar[d]&0\\
        0\ar[r]&\mu_N\ar[r]&J\ar[r]&J'\ar[r]&0
    }$$
    of exact sequences.
    There is an associated commutative diagram
    $$\xymatrix{
        \H^1(\Spec k,A')\ar[r]^\delta\ar[d]&\H^2(\Spec k,\mu_N)\ar@{=}[d]\\
        \H^1(S,J')\ar[r]^\delta&\H^2(S,\mu_N)
    }$$
    of boundary maps. If $X_\sigma$ is the $J'$-torsor associated to
    $\sigma\in\H^1(\Spec k,A')$ via the left vertical map, then commutativity of
    the diagram implies that $X$ splits $\delta(\sigma)=\beta$, as desired.
\end{proof}

\begin{remark}
    The entire proof goes through with $k$ replaced by a local ring $R$.
    Indeed, the only subtle point is the existence of
    a finite \'etale morphism $R\rightarrow R'$ such that
    $\beta$ is split by $R'$, which is proven
    in~\cite[Thm.~6.3]{auslander-goldman}.
\end{remark}

\subsection{Proof of Theorem~\ref{thm:a}}
\label{subsec:proof_A}

Let $E$ be an elliptic curve over a field $k$.  As in
Section~\ref{subsec:mu_n-isogeny}, let $\varphi\colon E \to E'$ be a
$\mu_N$-isogeny.  Then the short exact sequence \eqref{eq:liftpoint}
induces a boundary map $\H^1(\Spec k, \bZ/N) \to \H^{2}(\Spec k,
\mu_N)$. 
By Lemma~\ref{lem:coboundary}, this is given by taking the cup product
with $\delta(P)\in\H^1(\Spec k,\mu_N)$ where $P\in E'(k)/NE'(k)$ and
$\delta\colon E'(k)/NE'(k)\rightarrow\H^1(\Spec k,\mu_N)$.  The
homomorphism $\bZ/N \to E'$ induced from the dual isogeny determines a
map $\chi \mapsto X_\chi$ on cohomology $\H^1(\Spec k, \bZ/N) \to
\H^1(\Spec k, E')$.

Lemmas~\ref{lem:split_delta} and \ref{lem:delta_symbol} show that the
$E'$-torsor $X_\chi$ splits $\chi \cup \delta(P)=-\delta(P)\cup\chi$. We give in
Figure~\ref{fig:flambda} some discriminants of explicit families of
elliptic curves, i.e., elliptic curves $E$ defined over a rational
function field $k(\lambda)$, with exact order $N$-points $P$ together
with a calculation of explicit elements in $k[\lambda]$ that represent
$\delta(P) \in k(\lambda)^\times/k(\lambda)^{\times N}$. These
families are all defined over $\bZ$ away from the discriminant locus
(except when $N=2$ and we give two families) and they all admit points
over fields of every characteristic (again, except for
$N=2$). Verification of these properties is given in
Appendix~\ref{app}.

\begin{figure}[h]
    \begin{center}
        \footnotesize
        \begin{tabular}{|c|c|c|}
            \hline
            $N$ & $\Delta(\lambda)$ & $\delta(P)$\\
            \hline
            \hline
            $2$, $p\neq 2$ & $256\lambda^4-64\lambda^3$ & $\lambda$\\
            \hline
            $2$, $p=2$ & $\lambda^2$ & $\lambda$\\
            \hline
            $3$ & $(1-27\lambda)\lambda^3$ &$\lambda^2$\\
            \hline
            $4$ & $-16\lambda^5+\lambda^4$ & $\lambda^3$\\
            \hline
            $5$ &
            $\lambda^5(\lambda^2-11\lambda-1)$ & $\lambda^4$\\
            \hline
            $6$ &
            $\lambda^6(\lambda-1)^3(9\lambda-1)$ & $\lambda^5(\lambda-1)^4$\\
            \hline
            $7$ & 
            $-\lambda^7(\lambda-1)^7(\lambda^3+5\lambda^2-8\lambda+1)$ & $\lambda^6(\lambda-1)^3$\\
            \hline
            $8$ & 
            $\lambda^8(\lambda-1)^4(\lambda^2-6\lambda+1)/(\lambda+1)^{10}$ &$\lambda^7(\lambda-1)^6(\lambda+1)^4$\\ 
            \hline
            $9$ &
            $\lambda^9(\lambda+1)^9(\lambda^2+\lambda+1)^3(\lambda^3-3\lambda^2-6\lambda-1)$ & $\lambda^8(\lambda+1)^5(\lambda^2+\lambda+1)^6$\\
            \hline
            $10$ &
            $\lambda^{10}(\lambda+1)^{10}(2\lambda+1)^5(4\lambda^2+6\lambda+1)/(\lambda^2-\lambda-1)^{10}$ &$\lambda^9(\lambda+1)(2\lambda+1)^8(\lambda^2-\lambda-1)^5$\\
            \hline
            $12$ &
            $\lambda^{-24}(\lambda-1)^{12}(2\lambda-1)^6(3\lambda^2-3\lambda+1)^4(2\lambda^2-2\lambda+1)^3(6\lambda^2-6\lambda+1)$ & $-\lambda^{11}(\lambda-1)^{11}(2\lambda-1)^{10}(2\lambda^2-2\lambda+1)^8(3\lambda^2-3\lambda+1)^9$\\
            \hline
        \end{tabular}
    \end{center}
    \caption{The discriminants and boundary values $\delta(P)$ for various families of elliptic curves
    $E'$ with exact order $N$ points $P$. See the Appendix
    for the Weierstrass equations and details on the calculation.}
    \label{fig:flambda}
\end{figure}

\begin{proof}[Proof of Theorem~\ref{thm:a}]
  Assume that $N=2,3,4,$ or $5$ and that $\beta \in \H^2(\Spec k, \mu_N)$ is
    cyclic. Specifically, $\beta=\chi\cup u$ for some $\chi\in\H^1(\Spec
    k,\bZ/N)$ and $u\in k^\times$ representing an element of
    $k^\times/(k^\times)^N\iso\H^1(\Spec k,\mu_N)$.
  We give the proof of the theorem in the $N=5$ case; the proofs for $N=2,3,$ and $4$
  are similar. The $N=5$ line of Table~\ref{fig:flambda} describes the
    discriminant and boundary value $\delta(P)\in\H^1(\Spec k,\mu_N)$ of a family of elliptic
  curves with a fixed $5$-torsion point at $P=(0,0)$.  We claim that
  we can choose $\lambda \in k^\times$ such that
    \begin{itemize}
        \item $\lambda$ and $u$ generate the same subgroup of $k^\times/(k^\times)^5$ and
        \item $\Delta(\lambda)\neq 0$.
    \end{itemize}
    We could let $\lambda=u$ to arrange for the first condition, but we have to check also that
    $\Delta(u)=u^{5}(u^{2}-11 u-1)$ is non-zero. Since $u$ is a unit, this
    is equivalent to $u^2-11 u-1\neq 0$ in $k$. However, of course,
    sometimes this will vanish for particular $u$. Any number of the form
    $\lambda=uv^5$ where $v\in k^\times$ will also satisfy the first condition. If
    $k$ is infinite, this gives infinitely many possibilities for $\lambda$, at
    most $2$ of which will have vanishing discriminant, so we can find some
    number of the form $uv^5$ such that $\Delta(uv^5)\neq 0$. If $k$ is finite,
    then $\beta=0$ and there is nothing to prove. In any case, letting
    $\lambda=uv^5$ be a choice which satisfies the two criteria above, we find
    by Lemmas~\ref{lem:split_delta} and \ref{lem:delta_symbol} and by
    Figure~\ref{fig:flambda} that $X_\chi$
    splits $\chi\cup u^4$ and thus also $\chi\cup u$, as desired. 

    Now, fix $N=6,7,8,9,10,$ or $12$ and assume that $k$ is global.
    By the theorem of Albert--Brauer--Hasse--Noether
    (\cite{brauer_hasse_noether}, \cite{hasse:reciprocity}, see for
    example~\cite[XIII.6]{weil:basic_number_theory}, \cite{pierce}), every $\beta\in\H^2(\Spec k,\mu_N)$ is cyclic, of the form
    $\chi\cup u$ for some order $N$ character $\chi$ and some unit
    $u$.  Recall for instance from~\cite{tate-global} the fundamental exact sequence of
    class field theory
    $$0\rightarrow\Br(k)\rightarrow\bigoplus_{\mathfrak{p}}\Br(k_{\mathfrak{p}})\rightarrow\bQ/\bZ\rightarrow
    0,$$ where the direct sum ranges over all places of $k$, both
    archimedean and non-archimedean. We claim we can find $\lambda$
    such that the field extension $k(\delta(P)^{1/N})$ splits
    $\beta$. Let $S$ be the set of places $\mathfrak{p}$ in the
    support of $\beta$, i.e., where the associated Brauer class
    $\alpha$ is non-zero in $\Br(k_\pfrak)$. By standard ramification
    theory, in order for $k(\delta(P)^{1/N})$ to split $\beta$, it is
    enough to find $\lambda$ such that if $\mathfrak{p}\in S$ is
    finite, then $v_\pfrak(\delta(P))$ is a unit modulo $N$ (where we
    normalize so that $v_\pfrak$ takes integer values) and if
    $\mathfrak{p}\in S$ is a real place (and without loss of generality
    $N$ is even), then $\delta(P)$ is negative in the corresponding
    real embedding (since $\delta(P)$ is well-defined up to $N$th
    powers, for even $N$ the sign of $\delta(P)$ makes sense).  For
    $N = 6,7,8,9,10,$ or $12$, since
    $v_\pfrak(\delta(P)) \equiv (N-1)v_\pfrak(\lambda) \pmod{N}$ it suffices to
    choose $\lambda$ such that
    $v_\pfrak(\lambda)$ is a unit modulo $N$ for each finite
    place $\pfrak$ in $S$. For any real place $\pfrak$ in $S$, it
    suffices to take: $\lambda$ very negative for $N=6,8$, since
    $\delta(P)$ is an odd function with positive leading coefficient;
    $|\lambda-1|_\pfrak < \frac{1}{2}$ for
    $N=10$, since $\delta(P)$ is negative in that region; and
    $\lambda$ very large for $N=12$, since $\delta(P)$ is an even
    function with negative leading coefficient.  By weak
    approximation, we can arrange for all of these conditions to be
    satisfied by some $\lambda \in k$.

    It follows, in each of these cases, that $k(\delta(P)^{1/N})$ splits $\beta$. Now, theorems
    of Albert~\cite{albert} when $N$ is prime or Vishne~\cite{vishne} and
    Min\'{a}\v{c}--Wadsworth~\cite{minac-wadsworth}
    for $N$ composite (under the hypotheses of the theorem), imply that since $\beta$ is split by
    $k(\delta(P)^{1/N})$ we have in fact $\beta=\chi'\cup\delta(P)$ for some
    character $\chi'$. Thus, the curve $X_{\chi'}$ splits $\beta$, which is
    what we wanted to show.
\end{proof}

\subsection{Proof of Theorem~\ref{thm:b}}

\begin{proof}[Proof of Theorem~\ref{thm:b}]
    It is enough to know that under the conditions of the theorem, every class
    $\beta\in\H^2(\Spec k,\mu_N)$ is a sum of cyclic classes, for then we can
    split by a product of genus $1$ curves using Theorem~\ref{thm:a}. Using
    prime decomposition, we can separately handle $N=2,3,4$, and $5$. If the
    characteristic $p$ of $k$ is non-zero and divides $N$, then the result is
    due to Teichm\"uller; see~\cite[Thm.~9.1.4]{gille-szamuely}. So, assume
    that $p$ is prime to $N$. In this case, the fact results from Merkurjev's
    theorems~\cite{merkurjev-norm,merkurjev-fields} when $N=2$ or $3$, from the
    Merkurjev--Suslin theorem~\cite{merkurjev-suslin} when $N=4$
    and $k$ contains a primitive $4$th root of unity, and from a theorem of
    Matzri~\cite{matzri} when $N=5$.  See \cite[Section~3]{ketura} for
    further discussion.
\end{proof}

\subsection{Proof of Theorem~\ref{thm:c}}\label{sec:thmc}

We need the following easy lemma to begin.

\begin{lemma}\label{lem:pio}
    Suppose that $Y\in\H^1(S,E[N])$ is an $E[N]$-torsor with
    $Ob_\Lscr(Y)=\alpha \in \H^2(S,\Gm)$. If $X\in\H^1(S,E)$ is the $E$-torsor associated to
    $Y$, then $X$ splits $\alpha$.
\end{lemma}

\begin{proof}
  The group $E[N]$ acts on the flag
  $E[N]\subseteq E\subseteq\bP(p_*\Lscr)$. Thus, to the $E[N]$-torsor
  $Y$, we get a twisted form $P$ of $\bP(p_*\Lscr)$, a Severi--Brauer
  scheme over $S$
  variety with Brauer class $\alpha=Ob_\Lscr(Y)$, together with a flag
  of subvarieties $Y\subseteq X\subseteq P$ where $X$ is the genus
  $1$-curve associated to $Y$. Thus, $X$ splits the class $\alpha$.
\end{proof}

Note that Example~\ref{ex:bigprimes} shows that in general $X$ does not split the
$\mu_N$-gerbe $\beta$ associated to $\alpha$ in the context of Lemma~\ref{lem:pio}.

\begin{proof}[Proof of Theorem~\ref{thm:c}]
    Fix $k$ and $E$ as in the statement of the theorem and fix a full level $N$ structure
  $\varphi\colon E[N] \to \bZ/N\times\mu_N$, which we can take to be
  symplectic since we are working over a field. Let $P$ be the exact order $N$
    point $\varphi^{-1}(1\times 1)$ and let $\Lscr=\Oscr(0_E+P+\cdots+(N-1)P)$.
    By Proposition~\ref{prop:oneilerratum}, given an $E[N]$-torsor $Y$, which we
  can represent as a pair $\chi\in\H^1(\Spec k,\bZ/n)$ and
  $u\in\H^1(\Spec k,\mu_N)$, the period-index obstruction class
  $Ob_\Lscr(Y)$ is of the form $Ob_\Lscr(Y)=[\chi+\sigma,u]$, where
    $\sigma\in\H^1(\Spec k,\bZ/n)$ depends only
    on $E$, $N$, and $\phi$. In particular, if $Y$ corresponds to a pair
    $(\chi-\sigma,u)$, then $Ob_\Lscr(Y)=[\chi,u]$. By Lemma~\ref{lem:pio}, it follows that the
    $E$-torsor $X$ corresponding to $Y$ splits $[\chi,u]$, which completes the
    proof.
\end{proof}

\appendix
\section{Fisher's method}\label{app}

Fisher~\cite{fisher} has developed a method to compute $\delta(P)$, as
in Section~\ref{subsec:proof_A}, in the case of cyclic 5- and
7-isogenies.  After some correspondence, Fisher provided a helpful
explanation of his method, which we elaborate on here and carry out
for $N=4,\dotsc, 10, 12$ using \texttt{MAGMA} and Sutherland's modular
families~\cite{sutherland,sutherland:table}. This is used to populate
the table in Figure~\ref{fig:flambda}. Fisher's method works for
$N\geq 4$, hence for $N=2$ and $N=3$, we must utilize other
construction in the literature.

\subsection{Computing the boundary}

We recall the situation of Section~\ref{subsec:mu_n-isogeny}. Fix an
integer $N>1$. Let $E'$ be an elliptic curve over a field $k$ with a
rational point $P$ of exact order $N$, which generates a subgroup
$\ZZ/N\subseteq E'(k)$.  Let $E'\xrightarrow{\phi'}E$ be the isogeny
whose kernel is this subgroup. The kernel of the dual isogeny
$E\xrightarrow{\phi}E'$ is isomorphic to $\mu_N$ as a group scheme. We
want to compute the boundary map in the exact sequence
$$0\rightarrow\mu_N(k)\rightarrow E(k)\rightarrow E'(k)\xrightarrow{\delta} \H^1(\Spec
k,\mu_N)\iso k^\times/k^{\times N}.$$

\begin{lemma}
\label{lem:f}
    There is a rational function $f$ on $E'$ and a rational function $g$ on $E$
    such that $\mathrm{div}(f)=N(P)-N(0)$ and $f\circ\phi=g^N$.
\end{lemma}

\begin{proof}
    The existence of $f$ follows from the fact that the divisor $N(P)-N(0)$
    has degree $0$ and has sum $0$ in the group law on $E$. Now, let $Q$ be an
    $N$-torsion
    lift of $P$ to $E$, say over the algebraic closure. We have
    $\phi^*((P)-(0))=\sum_{R\in\ker(\phi)}((Q+R)-(R))$, which is the divisor of
    some function $g$ on $E$ since $NQ=0$. Now,
    $\mathrm{div}(f\circ\phi)=\mathrm{div}(g^N)$, so that after rescaling $f$
    by an element of $k^\times$, we can assume that $f\circ\phi=g^N$.
\end{proof}

\begin{lemma}\label{lem:coboundaryalg}
    Let $N\geq 3$ and let $f$ be a choice of rational function as in Lemma~\ref{lem:f}.
  Suppose that $P'\in E'(k)$ is a point not equal to $0$ or $P$, so
  that $f$ has neither a zero nor pole at $P'$. Then, $\delta(P')$ and
  $f(P')$ generate the same subgroup of $k^\times/k^{\times N}$. If
  moreover $P'=aP$ for some integer $a\in\{2,\ldots,N-1\}$, then
    $$\delta(P')\equiv f(P')$$ in $k^\times/k^{\times N}$.
\end{lemma}

\begin{proof}
  First, we prove that $\delta(P')$ and $f(P')$ generate the same
  subgroup of $k^\times/(k^\times)^N$.  We can assume that $N$ is a
  prime power and then assume that $N=p$, a prime, by suitably
  factoring the isogeny $\phi$.  The classes $\delta(P')$ and $f(P')$
  in $\H^1(\Spec k, \mu_p) \iso k^\times/k^{\times N}$ correspond to
  $\mu_p$-torsors $U$ and $V$ over $\Spec k$. These classes generate
  the same subgroup of $k^\times/k^{\times N}$ if and only if there is
  a $k$-isomorphism $U \to V$ (note that this $k$-isomorphism is
  $\mu_p$-equivariant if and only if the classes are equal).  It would
  be sufficient to have a $k$-morphism $U\rightarrow V$, which will
  then be an isomorphism since we are working with torsors for a prime
  order group. But, over $U$, the class $\delta(P')$ is zero, so that
  there is a lift along $\phi$ of $P'|_U$ to an element $Q'\in
  E(U)$. In this case, $f(P')=f(\phi(Q'))=g(Q')^N$, so that $f(P')$ is
  an $N$th-power in the group of units of $U$.  Thus $V$ admits a
  $U$-point, in other words, there is a morphism $U \to V$. Therefore,
  $\delta(P')\equiv f(P')^l$ for some $l\in(\bZ/N)^\times$.

  Now, we show that $l\equiv 1$ in $(\bZ/N)^\times$ when $P' = aP$.
  If $N$ is invertible in $k$, then this follows from a cocycle proof
  given in Fisher~\cite[Lem.~1.4]{fisher} or
  Silverman~\cite[Sec.~III.8]{silverman}. However, we can reduce to this case
  by lifting $E$ (which is ordinary) to characteristic $0$.
    Specifically, there is a $p$-complete, $p$-torsion-free discrete valuation
    ring $R$ with residue field $k$ (for example by~\cite[Thm.~29.1]{matsumura}) and a lift of $E$ to
    $R$ (lift a Weierstrass equation). The point $P$ also lifts by Hensel's
    lemma (applied to the moduli stack $\Mscr_1(N)$, which is a scheme) and the function $f$ does too. If $K$ denotes the fraction field of
    $R$, then we can compare $\delta(P')$ and $f(P')$ in the diagram
    $k^\times/(k^\times)^N\leftarrow R^\times/(R^\times)^N\hookrightarrow
    K^\times/(K^\times)^N$
    to deduce that $\delta(P')$ and $f(P')$ agree by reduction to the $N$
    invertible case.
\end{proof}

In order to compute the coboundary using
Lemma~\ref{lem:coboundaryalg}, one must find the function $f$ and
check that it is normalized correctly, so that $f\circ\phi=g^N$. A
completely general way to find the unnormalized function $f$ is via
Miller's algorithm, explained in~\cite[Sec.~11.8]{silverman}. But,
this is overkill for our present purposes, and we explain Fisher's
method for finding the unnormalized function $f$, which can be
implemented with a simple \texttt{MAGMA} script displayed in
Figure~\ref{fig:fisher}.

\begin{figure}[h]
    \begin{verbatim}
        > K<la> := FunctionField(Rationals());
        > E := EllipticCurve([1-la,-la,-la,0,0]);
        > V,Vmap := RiemannRochSpace(5*Divisor(E!0) - 5*Divisor(E![0,0]));
        > assert Dimension(V) eq 1;
        > KE<x,y> := FunctionField(E);
        > f := Vmap(V.1);
        > print f;
        (-x - 1)*y + x^2
    \end{verbatim}
    \caption{Fisher's \texttt{MAGMA} code, applied in the case of $N=5$ below.}
    \label{fig:fisher}
\end{figure}

Once a rational function $f$ on $E$ with the correct divisor
$N(P)-N(0)$ has been found, one normalizes it as follows. Let $uf$ be
a `correct' function, so that $(uf)\circ\phi=g^N$ for some rational
function $g$ on $E$. Here, $u$ is a unit in the base field $k$. Since
$\delta$ is a homomorphism, we can, for example, compute the
difference between $f(4P)$ and $f(2P)^2$, at least for $N\gg 0$.

Specifically, it is enough to find $uf$ up to $N$th powers, which we
do via the following lemma.

\begin{lemma}
    Fix an integer $N>2$.
    Choose integers
    $a,b,A,B$ such that the following hold:
    \begin{enumerate}
        \item[{\rm (a)}] $A+B\equiv 1\pmod{N}$,
        \item[{\rm (b)}] $aA+bB\equiv 0\pmod{N}$, and
        \item[{\rm (c)}] $a,b\neq 0,1\pmod{N}$.
    \end{enumerate}
    Then, $$u\equiv f(aP)^{-A}f(bP)^{-B}\pmod{k^{\times N}}.$$
\end{lemma}

\begin{proof}
  We have $$u^Af(aP)^A u^B f(bP)^B\equiv 1 \pmod{k^{\times N}}$$ by
  (b) and the fact that $n\mapsto (uf)(nP)$ for $n\neq 0,1\pmod{N}$
  are the values of a homomorphism, so after re-arranging and using
  (a), we find the desired claim using Lemma~\ref{lem:coboundaryalg},
  which applies by~(c).
\end{proof}

\begin{example}
The lemma will not apply when $N=2$ or $N=3$ because no choice of integers
$a,b,A,B$ satisfies (a), (b), and~(c) in those cases. We will need to argue via
a different approach, which we do individually below. For $N=4$, we can take $a=3$, $A=2$,
$b=2$, $B=-1$. For $N\geq 5$, we can take $a=2$, $A=2$, $b=4$, $B=-1$.
\end{example}

\subsection{Calculations}

\subsubsection*{Overview}

Each section below gives a family (two families for $N=2$) of elliptic
curves $E'$ with an exact order $N$ point $P$ together with a
calculation of $\delta(nP)$ for $n=1,\ldots,N-1$.  For $N\geq 4$ and
$n=2,\ldots,N-1$, these values are obtained by computing the
normalized function $f$ as described above, which we possibly rescale
by an $N$th power for notational convenience. Then, by writing
$P=mP+nP$ where $m,n\in\{2,\ldots,N-1\}$ and using that
$\delta(P)=\delta(mP)\delta(nP)$, we can fill in the final value
$\delta(P)$ of the table. As $N$ gets larger, the formulas are more
complicated and we include less information.  For $N=2$ and $N=3$, we
refer instead to the literature for appropriate families and boundary
map calculations.  With the exception of the families at $N=2$, the
remainder work in all characteristics.

The reader may verify the calculations here for $N\geq 4$ by running {\ttfamily
magma N-families.magma} with the attached {\ttfamily MAGMA} files, which computes the normalized function $f$ and the
values of $f$ on $2P,\ldots,(N-1)P$. Note however that we have at times chosen
to simplify the tables by changing the output of $f$ by $N$th powers.

\subsubsection*{$N=2$}

We use two families of elliptic curves $E'$, i.e., elliptic curves
over the rational function field $k(\lambda)$, depending on the
characteristic of $k$. The first, when $k$ has characteristic not 2,
is the family $$y^2=x^3-4\lambda x^2+\lambda x,$$ which has
discriminant $\Delta(\lambda)=256\lambda^4-64\lambda^3$ and a point
$P=(0,0)$ of exact order $2$.  Silverman
gives $$\delta(P)\equiv\lambda\pmod{k^{\times 2}}$$
in~\cite[Prop.~X.4.9]{silverman}.  The second, when $k$ has
characteristic $2$, is defined by
$$y^2+xy=x^3+x^2+\lambda^2,$$ which has discriminant
$\Delta(\lambda)=\lambda^2$ and a point $P=(0,\lambda)$ of exact order $2$.
Kramer gives $$\delta(P) \equiv\lambda\pmod{k^{\times 2}}$$
in~\cite[Prop.~1.1(b)]{kramer}.

\subsubsection*{$N=3$}

Kozuma~\cite{kozuma} studies the elliptic curve $E'_\lambda$ over
$k(\lambda)$ defined by
$$y^2+xy+\lambda y=x^3$$ with an exact order $3$ point $P=(0,0)$. The
discriminant of this Weierstrass equation is $\Delta(\lambda)=(1-27\lambda)\lambda^3$.
Figure~\ref{fig:3} shows the values of the boundary map at the multiples of
$P$, which can be found in~\cite[Eq.~3.5]{kozuma}.

The discriminant is not identically zero modulo any prime $p$, so this family
and the corresponding calculation is available in all characteristics.

\begin{figure}[h]
    \begin{center}
        \begin{tabular}{|c|c|c|}
            \hline
            $nP$ & $(x_{nP},y_{nP})$ & $\delta(nP)$\\
            \hline
            $P$ & $(0,0)$ & $\lambda^2$\\
            $2P$ & $(0,-\lambda)$ & $\lambda$\\
            \hline
        \end{tabular}
    \end{center}
    \caption{The coordinates of the multiples of $P$ and the values of
    $\delta(nP)$ for $N=3$.}
    \label{fig:3}
\end{figure}

\subsubsection*{$N=4$}

For $N \geq 4$ we will use families of elliptic curves in Tate normal
form
\begin{equation}\label{eq:tate}
y^2+(1-c)xy-by=x^3-bx^2
\end{equation}
for values of $b,c \in k(\lambda)$.  For $N=4$, we take $c=0$, as
explained in \cite[V.5~(5.31)]{knapp}, and set $b=-\lambda$ to get the
elliptic curve
$$y^2+xy+\lambda y=x^3+\lambda x^2$$
over $k(\lambda)$, with a point of order $4$ at $P=(0,0)$ and discriminant  $\Delta(\lambda)=-16\lambda^5+\lambda^4$.

As an example of using Fisher's method, we start with the rational
function $f(x,y)=y-x^2$ having the correct divisor. However, there is
a normalization problem in this case.  Indeed, we have
$f(3P)=f(0,-\lambda)=-\lambda$, which implies that we should set
$f(P)=-\lambda^3$ (or $f(P)=-\tfrac{1}{\lambda}$). However,
$f(2P)=f(-\lambda,0)=-\lambda^2$ and we should
have $$\lambda^3=f(2P)f(3P)\equiv f(P)=-\lambda^3,$$ which is
typically false. So, following Fisher's algorithm, we can correct $f$
by multiplying by $-1$. Having thus found $f = x^2-y$, we can compute
$\delta(2P)=f(2P)$ and $\delta(3P)=f(3P)$, which we can use to find
$\delta(P)\equiv\delta(2P)\delta(3P)\pmod{k(\lambda)^{\times^4}}$.

\begin{figure}[h]
    \begin{center}
        \begin{tabular}{|c|c|c|}
            \hline
            $nP$ & $(x_{nP},y_{nP})$ & $\delta(nP)$\\
            \hline
            $P$ & $(0,0)$ & $\lambda^3$\\
            $2P$ & $(-\lambda,0)$ & $\lambda^2$\\
            $3P$ & $(0,-\lambda)$ & $\lambda$\\
            \hline
        \end{tabular}
    \end{center}
    \caption{The coordinates of the multiples of $P$ and the values of
    $\delta(nP)$ (equal to $f(nP)$ for $n=2,3$ where $f=x^2-y$) for $N=4$.}
    \label{fig:4}
\end{figure}

\subsubsection*{$N=5$}

We use the Tate normal form \eqref{eq:tate} with $b=c=\lambda$, as
in \cite[V.5~(5.31)]{knapp}, to get the elliptic curve
$$y^2+(1-\lambda)xy-\lambda y=x^3-\lambda x^2$$
over $k(\lambda)$ with a point of order $5$ point $P=(0,0)$ and
discriminant $\Delta(\lambda)=\lambda^7-11\lambda^6-\lambda^5$,
cf.~\cite{fisher}. Fisher's method produces the rational function
$f(x,y)=-x^2+xy+y$, which is already normalized (up to $5$th powers).

\begin{figure}[h]
    \begin{center}
        \begin{tabular}{|c|c|c|}
            \hline
            $nP$ & $(x_{nP},y_{nP})$ & $\delta(nP)$\\
            \hline
            $P$ & $(0,0)$ & $\lambda^4$\\
            $2P$ & $(\lambda,\lambda^2)$ & $\lambda^3$\\
            $3P$ & $(\lambda,0)$ & $-\lambda^2$\\
            $4P$ & $(0,\lambda)$ & $\lambda$\\
            \hline
        \end{tabular}
    \end{center}
    \caption{The coordinates of the multiples of $P$ and the values of
    $\delta(nP)$ (equal to $f(nP)$ for $n=2,3,4$, where
    $f=-x^2+xy+y$) for $N=5$.}
    \label{fig:5}
\end{figure}

\subsubsection*{$N=6$}

For $N \geq 6$, we will utilize the Tate normal form \eqref{eq:tate},
with the convention that 
\begin{equation}\label{eq:suth}
c = s(r-1),\qquad  b=rc
\end{equation}
for elements $r, s \in k(\lambda)$, as in
Sutherland~\cite{sutherland,sutherland:table}.  For $N=6$, we use
$r=1-\lambda$ and $s=1$ giving the elliptic curve
$$y^2+(1+\lambda)xy-(\lambda-\lambda^2) y=x^3-(\lambda-\lambda^2) x^2$$
over $k(\lambda)$ with an order $6$ point $P=(0,0)$ and discriminant
$\Delta(\lambda)=\lambda^6(\lambda-1)^3(9\lambda-1)$.


\begin{figure}[h]
    \begin{center}
        \begin{tabular}{|c|c|c|}
            \hline
            $nP$ & $(x_{nP},y_{nP})$ & $\delta(nP)$\\
            \hline
            $P$ & $(0,0)$ & $\lambda^5(\lambda-1)^4$\\
            $2P$ & $(\lambda(\lambda-1),-\lambda^2(\lambda - 1))$ &
            $\lambda^4(\lambda-1)^2$\\
            $3P$ & $(-\lambda,\lambda^2)$ &
            $\lambda^3$\\
            $4P$ & $(\lambda(\lambda-1),0)$ & $\lambda^2(\lambda-1)^4$\\
            $5P$ & $(0,\lambda(\lambda-1))$ & $\lambda(\lambda-1)^2$\\
            \hline
        \end{tabular}
    \end{center}
    \caption{The coordinates of the multiples of $P$ and the values of
    $\delta(nP)$ (equal to $f(nP)$ for $n=2,3,4,5$, where $f=-2xy-(1-\lambda) y+x^3
    + (1-\lambda) x^2$) for $N=6$.}
    \label{fig:6}
\end{figure}

\subsubsection*{$N=7$}

We use the Tate normal form \eqref{eq:tate} with convention
\eqref{eq:suth} $r=s=1-\lambda$, to get the
elliptic curve
$$y^2+(1+\lambda-\lambda^2)xy+\lambda(1-\lambda)^2
y=x^3+\lambda(1-\lambda)^2x^2$$ 
over $k(\lambda)$ with an order 7 point $P=(0,0)$ and
discriminant
$\Delta(\lambda)=-\lambda^7(\lambda-1)^7(\lambda^3+5\lambda^2-8\lambda+1)$,
cf.~\cite{fisher}.

\begin{figure}[h]
    \begin{center}
        \begin{tabular}{|c|c|c|}
            \hline
            $nP$ & $(x_{nP},y_{nP})$ & $\delta(nP)$\\
            \hline
            $P$ & $(0,0)$ & $\lambda^6(\lambda-1)^3$\\
            $2P$ & $(-\lambda(\lambda-1)^2,-\lambda^2(\lambda-1)^3)$ &
            $\lambda^5(\lambda-1)^6$\\
            $3P$ & $(\lambda(\lambda-1),-\lambda^2(\lambda-1))$ &
            $\lambda^4(\lambda-1)^2$\\
            $4P$ & $(\lambda(\lambda-1), \lambda^2(\lambda-1)^2)$ &
            $-\lambda^3(\lambda-1)^5$\\
            $5P$ & $(-\lambda(\lambda-1)^2,0)$ & $\lambda^2(\lambda-1)^8$\\
            $6P$ & $(0, -\lambda(\lambda-1)^2)$ & $\lambda(\lambda-1)^4$\\
            \hline
        \end{tabular}
    \end{center}
    \caption{The coordinates of the multiples of $P$ and the values of
    $\delta(nP)$ (equal to $f(nP)$ for $n=2,3,4,5,6$, where $f=-x^2y +
    (2\lambda-3)xy - (\lambda-1)^2y + \lambda x^3 + (\lambda-1)^2 x^2$) for
    $N=7$.}
    \label{fig:7}
\end{figure}

\subsubsection*{$N=8$}

We use the Tate normal form \eqref{eq:tate} with convention
\eqref{eq:suth} $r=1/(1+\lambda)$ and $s=1-\lambda$ yielding the
elliptic curve
$$
y^2+\bigl(1-\tfrac{\lambda(\lambda-1)}{\lambda+1}\bigr)xy-\tfrac{\lambda(\lambda-1)}{\lambda+1}
y=x^3-\tfrac{\lambda(\lambda-1)}{\lambda+1} x^2$$ over $k(\lambda)$
with an order $8$ point $P=(0,0)$ and discriminant
$\Delta(\lambda)=\lambda^8(\lambda-1)^4(\lambda^2-6\lambda+1)/(\lambda+1)^{10}$, cf.~\cite{sutherland,sutherland:table}.

\begin{figure}[h]
    \begin{center}
        \begin{tabular}{|c|c|c|}
            \hline
            $nP$ & $(x_{nP},y_{nP})$ &$\delta(nP)$\\
            \hline
            $P$  & $(0,0)$ & $\lambda^7(\lambda-1)^6(\lambda+1)^4$\\
            $2P$ & $(\lambda(\lambda-1)/(\lambda+1)^2,
                   \lambda^2(\lambda-1)^2/(\lambda+1)^3)$ & $\lambda^6(\lambda-1)^4$\\
            $3P$ & $(\lambda(\lambda-1)/(\lambda+1),
                   -\lambda^2(\lambda-1)/(\lambda+1)^2)$ & $\lambda^5(\lambda-1)^2(\lambda+1)^4$\\
            $4P$ & $(-\lambda/(\lambda+1)^2,
                   \lambda^2/(\lambda+1)^3)$ & $\lambda^4$\\
            $5P$ & $(\lambda(\lambda-1)/(\lambda+1),
                   \lambda^2(\lambda-1)^2/(\lambda+1)^2)$ & $\lambda^3(\lambda-1)^6(\lambda+1)^4$\\
            $6P$ & $(\lambda(\lambda-1)/(\lambda+1)^2,0)$ & $\lambda^2(\lambda-1)^4$\\
            $7P$ & $(0,\lambda(\lambda-1)/(\lambda+1)^2)$ & $\lambda(\lambda-1)^2(\lambda+1)^4$\\
            \hline
        \end{tabular}
    \end{center}
    \caption{The value of $\delta(nP)$ for $N=8$.}
    \label{fig:8}
\end{figure}

\subsubsection*{$N=9$}

We use the Tate normal form \eqref{eq:tate} with convention
\eqref{eq:suth} $r=\lambda^2+\lambda+1$ and $s=\lambda+1$ yielding an
elliptic curve over $k(\lambda)$ with an order $9$ point $P=(0,0)$ and
discriminant
$$\Delta(\lambda)=\lambda^9(\lambda+1)^9(\lambda^2+\lambda+1)^3(\lambda^3-3\lambda^2-6\lambda-1).$$

\begin{figure}[h]
    \begin{center}
        \begin{tabular}{|c|c|c|}
            \hline
            $nP$ & $(x_{nP},y_{nP})$ & $\delta(nP)$\\
            \hline
            $P$  & $(0,0)$ & $\lambda^8(\lambda+1)^5(\lambda^2+\lambda+1)^6$\\
            $2P$ & $(\lambda(\lambda+1)^2(\lambda^2+\lambda+1),\lambda^2(\lambda+1)^4(\lambda^2+\lambda+1))$ & $\lambda^7(\lambda+1)^{10}(\lambda^2+\lambda+1)^3$\\
            $3P$ & $(\lambda(\lambda+1)^2,\lambda^2(\lambda+1)^3)$ & $-\lambda^6(\lambda+1)^6$\\
            $4P$ & $(\lambda(\lambda+1)(\lambda^2+\lambda+1),\lambda^2(\lambda+1)(\lambda^2+\lambda+1)^2)$ & $\lambda^5(\lambda+1)^2(\lambda^2+\lambda+1)^6$\\
            $5P$ & $(\lambda(\lambda+1)(\lambda^2+\lambda+1),\lambda^2(\lambda+1)^2(\lambda^2+\lambda+1))$ & $-\lambda^4(\lambda+1)^7(\lambda^2+\lambda+1)^3$\\
            $6P$ & $(\lambda(\lambda+1)^2,\lambda^2(\lambda+1)^4)$ & $\lambda^{3}(\lambda+1)^3$\\
            $7P$ & $(\lambda(\lambda+1)^2(\lambda^2+\lambda+1),0)$ & $-\lambda^2(\lambda+1)^8(\lambda^2+\lambda+1)^6$\\
            $8P$ & $(0,\lambda(\lambda+1)^2(\lambda^2+\lambda+1))$ & $\lambda(\lambda+1)^4(\lambda^2+\lambda+1)^3$\\
            \hline
        \end{tabular}
    \end{center}
    \caption{The values of $\delta(nP)$ for $N=9$.}
    \label{fig:9}
\end{figure}

\subsubsection*{$N=10$}

We use the Tate normal form \eqref{eq:tate} with convention
\eqref{eq:suth} $r=-(\lambda+1)^2/(\lambda^2-\lambda-1)$ and $s=\lambda+1$ yielding an
elliptic curve over $k(\lambda)$ with an order $10$ point $P=(0,0)$ and
discriminant
$$\Delta(\lambda)=\lambda^{10}(\lambda+1)^{10}(2\lambda+1)^5(4\lambda^2+6\lambda+1)/(\lambda^2-\lambda-1)^{10}.$$

\begin{figure}[h]
    \begin{center}
        \begin{tabular}{|c|c|}
            \hline
            $nP$ & $\delta(nP)$\\
            \hline
            $P$ &
            $\lambda^9(\lambda+1)(2\lambda+1)^8(\lambda^2-\lambda-1)^5$\\
            $2P$ & $\lambda^8(\lambda+1)^2(2\lambda+1)^6$\\
            $3P$ & $\lambda^7(\lambda+1)^3(2\lambda+1)^4(\lambda^2-\lambda-1)^5$\\
            $4P$ & $\lambda^6(\lambda+1)^4(2\lambda+1)^2$\\
            $5P$ & ${\lambda^5(\lambda+1)^5}{(\lambda^2-\lambda-1)^5}$\\
            $6P$ & $\lambda^4(\lambda+1)^6(2\lambda+1)^8$\\
            $7P$ & ${\lambda^3(\lambda+1)^7(2\lambda+1)^6}{(\lambda^2-\lambda-1)^5}$\\
            $8P$ & $\lambda^2(\lambda+1)^8(2\lambda+1)^4$\\
            $9P$ & ${\lambda(\lambda+1)^9(2\lambda+1)^2}{(\lambda^2-\lambda-1)^5}$\\
            \hline
        \end{tabular}
    \end{center}
    \caption{The values of $\delta(nP)$ for $N=10$.}
    \label{fig:10}
\end{figure}

\subsubsection*{$N=12$}

We use the Tate normal form \eqref{eq:tate} with convention
\eqref{eq:suth} $r=(2\lambda^2-2\lambda+1)/\lambda$ and
$s=(3\lambda^2-3\lambda+1) / \lambda^2$ yielding an
elliptic curve over $k(\lambda)$ with an order $12$ point $P=(0,0)$ and
discriminant
$$\Delta(\lambda)=\lambda^{-24}(\lambda-1)^{12}(2\lambda-1)^6(3\lambda^2-3\lambda+1)^4(2\lambda^2-2\lambda+1)^3(6\lambda^2-6\lambda+1).$$

\begin{figure}[h]
    \begin{center}
        \begin{tabular}{|c|c|}
            \hline
            $nP$ & $\delta(nP)$\\
            \hline
            $P$&
            $-\lambda^{11}(\lambda-1)^{11}(2\lambda-1)^{10}(2\lambda^2-2\lambda+1)^8(3\lambda^2-3\lambda+1)^9$\\
            $2P$ & $\lambda^{10}(\lambda-1)^{10}(2\lambda-1)^8(2\lambda^2-2\lambda+1)^4(3\lambda^2-3\lambda+1)^6$\\
            $3P$ &
            $-\lambda^9(\lambda-1)^9(2\lambda-1)^6(3\lambda^2-3\lambda+1)^3$\\
            $4P$ &
            $\lambda^8(\lambda-1)^8(2\lambda-1)^4(2\lambda^2-2\lambda+1)^8$\\
            $5P$ &
            $-\lambda^7(\lambda-1)^{7}(2\lambda-1)^{2}(2\lambda^2-2\lambda+1)^{4}(3\lambda^2-3\lambda+1)^{9}$\\
            $6P$ &
            $\lambda^6(\lambda-1)^6(3\lambda^2-3\lambda+1)^6$\\
            $7P$ &
            $-\lambda^5(\lambda-1)^{5}(2\lambda-1)^{10}(2\lambda^2-2\lambda+1)^8(3\lambda^2-3\lambda+1)^{3}$\\
            $8P$ &
            $\lambda^4(\lambda-1)^4(2\lambda-1)^8(2\lambda^2-2\lambda+1)^4$\\
            $9P$ &
            $-\lambda^3(\lambda-1)^3(2\lambda-1)^6(3\lambda^2-3\lambda+1)^9$\\
            $10P$ &
            $\lambda^2(\lambda-1)^2(2\lambda-1)^4(2\lambda^2-2\lambda+1)^8(3\lambda^2-3\lambda+1)^6$\\
            $11P$ &
            $-\lambda(\lambda-1)(2\lambda-1)^2(2\lambda^2-2\lambda+1)^4(3\lambda^2-3\lambda+1)^3$\\
            \hline
        \end{tabular}
    \end{center}
    \caption{The values of $\delta(nP)$ for $N=12$.}
    \label{fig:12}
\end{figure}

\small
\bibliographystyle{amsplain}
\bibliography{g1c}

\vspace{20pt}
\scriptsize
\noindent
Benjamin Antieau\\
Department of Mathematics\\
Northwestern University\\
2033 Sheridan Road\\
Evanston, IL 60208\\
\texttt{antieau@northwestern.edu}

\vspace{20pt}
\noindent
Asher Auel\\
Department of Mathematics\\
Dartmouth College\\
6188 Kemeny Hall\\
Hanover, NH 03755\\
\texttt{asher.auel@dartmouth.edu}

\end{document}